\documentclass[11pt]{article}

\usepackage{amssymb,amsthm,amsmath,hyperref,txfonts}
\usepackage{color}
\usepackage{graphicx,float}
\usepackage{indentfirst}
\numberwithin{equation}{section}

\newtheorem{thm}{Theorem}[section]
\newtheorem{lem}{Lemma}[section]

\newtheorem{rem}{Remark}[section]

\newcommand{\beq}{\begin{eqnarray}}
\newcommand{\eeq}{\end{eqnarray}}
\newcommand{\beqno}{\begin{eqnarray*}}
\newcommand{\eeqno}{\end{eqnarray*}}
\newcommand{\be}{\begin{equation}}
\newcommand{\ee}{\end{equation}}
\theoremstyle{definition}
\newtheorem{defn}{Definition}[section]

\allowdisplaybreaks

\begin{document}
\title{\bf Hydrodynamic limit for compressible Navier-Stokes-Vlasov-Poisson equations with local alignment force}
\author{ Yunfei Su $^{1}$\thanks{E-mail:
		suyfmath1003@163.com},\ \
	Lei Yao$^{2}$\thanks{ E-mail:
		yaolei1056@hotmail.com}
	\\
	\textit{\small
		1. School of Mathematics, North University of China, Taiyuan  030051, P.R. China} \\
	\textit{\small 2.
		School of Mathematics and Statistics, Northwestern Polytechnical University, Xi'an 710129, P.R. China }}
\date{}
\maketitle
\begin{abstract}
We investigate the  hydrodynamic limit of weak solutions to compressible Navier-Stokes-Vlasov-Poisson equations with local alignment force in three-dimensional torus  domain.  Due to the absence of dissipation terms in particle equation, it is difficult to study this problem.  Based on  the relative entropy method, it is shown that the global weak solutions of the compressible Navier-Stokes-Vlasov-Poisson equations converge to the smooth solutions of the limiting two-phase fluid model.
We obtained that the distribution function $f^{\epsilon}$ converges to a Dirac distribution in velocity, the fluid density $\rho^{\epsilon}$ and velocity $u^{\epsilon}$ converge  to $\rho$ and $u$, respectively.  
\vspace{4mm}


\end{abstract}

\section{Introduction}
We consider the compressible Navier-Stokes-Vlasov-Poisson equations in $\mathbb{R}^{+}\times\mathbb{T}^{3}\times\mathbb{R}^{3}$:
\begin{equation}\label{O-eq1.1-1}
\left\{
\begin{array}{l}
\partial_{t}f+\xi\cdot\nabla_{x}f+\mathrm{div}_{\xi}((u-\xi-\nabla_{x}\Phi)f)=-\mathrm{div}_{\xi}((u_{f}-\xi)f),\\
-\Delta_{x}\Phi=\rho_{f}-1,\\
\partial_{t}\rho+\mathrm{div}_{x}(\rho u)=0,\\
\partial_{t}(\rho u)+\mathrm{div}_{x}(\rho u \otimes u)+\nabla_{x}p-\Delta_{x} u
=-\int_{\mathbb{R}^{3}}(u-\xi)f~d\xi,\\
\end{array}
\right.
\end{equation}
where $x\in\mathbb{T}^{3}$ is the spatial variable,  $t\in\mathbb{R}^{+}$ is the time variable, $\xi\in \mathbb{R}^{3}$ is the velocity of the particles.
This system of nonlinear PDEs describes the dynamics of charged particles immersed in a compressible viscous fluid and belongs to the broad family of the fluid-particle system.
  The motion of numerous small solid charged particles  can be modeled by the Vlasov-Poisson equations and compressible viscous fluid can be described by the  Navier-Stokes equations.
Here,
 the unknown function $f(t,x,\xi )$  defined in the phase space $\mathbb{T}^{3} \times \mathbb{R}^{3}$ at time $t\in\mathbb{R}^{+}$ denotes the distribution function of the particles,
$\rho(t,x)\geq0$,  $u(t,x)$ and  $p(t,x)$ denote  density, velocity field and pressure of fluids, respectively. $\Phi(t,x)$ is the potential of the electric field generated by the charged particles.
The coupling is made of a drag term in the Vlasov equation $\eqref{O-eq1.1-1}_{1}$, and of a source term in the Navier-Stokes equation $\eqref{O-eq1.1-1}_{3}-\eqref{O-eq1.1-1}_{4}$, called the Brinkman force, which is used to describe the exchange of momentum between the particles and the fluid.

   Besides,  $\rho_{f}(t,x)$ and $u_{f}(t,x)$ denote the local averaged particle density and velocity, respectively, i.e.,
   \begin{equation*}
   \rho_{f}(t,x):=\int_{\mathbb{R}^{3}}f(t,x,\xi)~d\xi,~~~u_{f}(t,x):=\frac{1}{\rho_{f}(t,x)}\int_{\mathbb{R}^{3}}\xi f(t,x,\xi)~d\xi.
   \end{equation*}
   Motivated by Goudon, Jabin and Vasseur \cite{Goudon2004a,Goudon2004} for a dimensionless analysis, we can rewrite
   (1.1) into the dimensionless form as follows:
  \begin{equation}\label{O-eq1.1}
  \left\{
  \begin{array}{l}
  \partial_{t}f^{\epsilon}+\xi\cdot\nabla_{x}f^{\epsilon}+\mathrm{div}_{\xi}((u^{\epsilon}-\xi-\nabla_{x}\Phi^{\epsilon})f)=-\frac{1}{\epsilon}\mathrm{div}_{\xi}((u^{\epsilon}_{f}-\xi)f^{\epsilon}),\\
  -\Delta_{x}\Phi^{\epsilon}=\rho^{\epsilon}_{f}-1,\\
  \partial_{t}\rho^{\epsilon}+\mathrm{div}_{x}(\rho^{\epsilon} u^{\epsilon})=0,\\
  \partial_{t}(\rho^{\epsilon} u^{\epsilon})+\mathrm{div}_{x}(\rho^{\epsilon} u^{\epsilon} \otimes u^{\epsilon})+\nabla_{x}p^{\epsilon}-\Delta_{x} u^{\epsilon}
  =-\int_{\mathbb{R}^{3}}(u^{\epsilon}-\xi)f^{\epsilon}~d\xi.
  \end{array}
  \right.
  \end{equation}
   The system \eqref{O-eq1.1} subjects to the following initial conditions:
   \begin{align}\label{O-ini1}
   &f^{\epsilon}(0,x,\xi)=f^{\epsilon}_{0}(x,\xi)\geq 0,~~~\Phi^{\epsilon}(0,x)=\Phi_{0}^{\epsilon}(x),~~~\rho^{\epsilon}(0,x)=\rho^{\epsilon}_{0}(x),~~~ u^{\epsilon}(0,x)=u^{\epsilon}_{0}(x).
   \end{align}

   In this paper, we prove the hydrodynamic limit of weak solutions to the problem  \eqref{O-eq1.1}-\eqref{O-ini1}. Motivated by Choi-Jung \cite{Choi2023}, we firstly notice that the potential $\Phi$ can be represented by using the interaction potential $K$ which satisfies the following conditions:

  (1) The potential $K$ is an even function which can be written as $K(x)=c_{1}|x|^{-1}+G_{1}(x)$,
  where $c_{1}>0$ is normalization constant and $G_{1}$ is smooth function over $\mathbb{T}^{3}$.

  (2) For any $h\in L^{2}(\Omega)$ with $\int_{\mathbb{T}^{3}} h ~dx=0$, $\Phi:=K*h\in H^{1}(\mathbb{T}^{3})$
  is the unique function that satisfies the following conditions:
  $$\int_{\mathbb{T}^{3}} \Phi ~dx=0,$$
  and
  $$\int_{\mathbb{T}^{3}} \nabla_{x}\Phi\cdot\nabla_{x}\psi ~dx=\int_{\mathbb{T}^{3}} h\psi~dx,$$
  for any $\psi\in H^{1}(\mathbb{T}^{3})$, i.e., $\Phi$ is the unique weak solution to $-\Delta \Phi=h.$
  Then the system \eqref{O-eq1.1} can be  rewritten  as follows:
  \begin{equation}\label{O-written-eq1.1}
  \left\{
  \begin{array}{l}
  \partial_{t}f^{\epsilon}+\xi\cdot\nabla_{x}f^{\epsilon}+\mathrm{div}_{\xi}((u^{\epsilon}-\xi-\nabla_{x}K*(\rho^{\epsilon}_{f}-1))f^{\epsilon})=-\frac{1}{\epsilon}\mathrm{div}_{\xi}((u^{\epsilon}_{f}-\xi)f^{\epsilon}),\\
  \partial_{t}\rho^{\epsilon}+\mathrm{div}_{x}(\rho^{\epsilon} u^{\epsilon})=0,\\
  \partial_{t}(\rho^{\epsilon} u^{\epsilon})+\mathrm{div}_{x}(\rho^{\epsilon} u^{\epsilon}\otimes u^{\epsilon})+\nabla_{x}p^{\epsilon}-\Delta_{x} u^{\epsilon}
  =-\int_{\mathbb{R}^{3}}(u^{\epsilon}-\xi)f^{\epsilon}~d\xi.\\
  \end{array}
  \right.
  \end{equation}

  In the past decades years, the existence of global solutions to the fluid-particle model has been investigated extensively. For the fluid-particle model with the local alignment force,
  Choi and Jung \cite{Choi2021} obtained the global existence of weak solutions to the compressible Navier-Stokes-Vlasov-Fokker-Planck equations with homogeneous Dirichlet boundary condition for the fluid velocity and the specular reflection boundary condition for the kinetic equation. And later,
  Li and Li \cite{Li2021} studied the global existence of  weak solutions for the Vlasov-Fokker-Planck equation coupled with the compressible Navier-Stokes equations with nonhomogeneous Dirichlet boundary condition.
  Choi and Jung \cite{Choi2023} proved the global existence of weak solutions to the incompressible Navier-Stokes-Vlasov-Poisson(-Fokker-Planck) equation on $\mathbb{T}^{d}$($d\geq2$).
  When  the local alignment force is absent and the particles are not charged to the fluid-particle model, there are many meaningful results about the global existence of solutions to the Vlasov(-Fokker-Planck) equations coupled with the incompressible/inhomogeneous incompressible Navier-Stokes equations, cf. \cite{Boudin2009,Boudin2017,Chae2011,Choi2015,Goudon2010,Hamdache1998,Wang2015,Yu2013} and coupled with the compressible Navier-Stokes equations, cf.
  \cite{Li2017,Mellet2007}. When the dynamics of charged particles immersed in an incompressible viscous fluid,  Anoshchenko, Khruslov and Stephan
  \cite{Anoshchenko2010} studied the global existence of weak solutions to the incompressible Navier-Stokes-Vlasov-Poisson equations in the bounded domain; furthermore, the first two authors and Iegorov \cite{Anoshchenko2014} proved the  global weak solutions to the incompressible Navier-Stokes-Vlasov-Fokker-Planck-Poisson equations. Chen, Li, Li and Zamponi \cite{Chen2023} obtained the global existence of  weak solutions to the compressible Navier-Stokes-Vlasov-Poisson-Fokker-Planck system in a three-dimensional bounded domain with nonhomogeneous Dirichlet boundary conditions.  The mathematical analysis of other related models, such as Euler-Vlasov equations \cite{Baranger2006},  the Navier-Stokes-Vlasov-Boltzmann(Poisson) \cite{Cui2021,Gamba2020,Yao2018}, MHD/Vlasov-Fokker-Planck equations and Euler-Maxwell-Vlasov-Fokker-Planck equations \cite{Jiang2017,Jiang2020} and so on, which have also received
  much attentions now.

   There are also some results about  the hydrodynamic limit of
   weak solutions to the fluid-particle model. When the local alignment
   force is involved in the kinetic equation, i.e.,  the standard Cucker-Smale alignment term and the Motsch-Tadmor alignment term are involved.
Karper, Mellet and Trivisa \cite{Karper2015}   introduced a strong local
   alignment interaction as the singular limit of the alignment proposed by Motsch-Tadmor \cite{Motsch2011}, who improved the Cucker-Smale model by considering new interaction, which is non-local and non-symmetric alignment. Based on the relative entropy method and the weak compactness argument,
   Kang and Vasseur \cite{Kang2015} studied that   weak solution  of the Vlasov-type equation converges to strong solution  of the pressureless Euler system.
   Carrillo, Choi and Karper \cite{Carrillo2016} considered the hydrodynamic limit of weak solutions
   to a kinetic flocking model coupled with the incompressible Navier-Stokes equations  on $\mathbb{T}^{3}$; After that,
    the result of \cite{Carrillo2016} was  extended by Choi and Jung \cite{Choi[2020]copyright2020} to  the compressible case with the density-dependent viscosity.
   Recently, Choi and Jung \cite{Choi2021}
   proved the hydrodynamic limit of weak solutions to the compressible Navier-Stokes-Vlasov-Fokker-Planck equations in a three-dimensional bounded domain, where the fluid velocity satisfies the homogeneous Dirichlet boundary condition and the particle distribution function satisfies the reflection boundary condition. Shi-Su-Yao \cite{Shi2023} studied the asymptotic analysis of weak solutions to the 1D compressible {N}avier-{S}tokes-{V}lasov equations with local alignment force. For other related  fluid-particle model,
   such as the Cucker-Smale model \cite{Figalli2019,Karper2015}, the incompressible(compressible) Navier-Stokes-Cucker-Smale equations \cite{Bae2012,Bae2014} and so on.
   For the fluid-particle model without the local alignment force,
    Goudon \cite{Goudon2001} proved the hydrodynamic limit and stratified limit to 1D viscous Burgers-Vlasov equations. For the multi-dimensional case, Goudon, Jabin and Vasseur \cite{Goudon2004a,Goudon2004} considered two classes of hydrodynamic
   limit to the incompressible Navier-Stokes-Vlasov-Fokker-Planck equations; Mellet and
   Vasseur \cite{Mellet2008} further studied the compressible case. Lately,
   the result of \cite{Goudon2004} and \cite{Mellet2008} was  extended by Su and Yao \cite{Su2020} to  the inhomogeneous incompressible  case.
   Recently, Han-Kwan and Michel \cite{Han-Kwan2024} proved
   the two classes of hydrodynamic limits for the incompressible Navier-Stokes-Vlasov equations
   on $\mathbb{T}^{3}$, referred to as the light particles and  fine particle regimes. Su-Wu-Yao-Zhang \cite{Su2023} established that the hydrodynamic limit of weak solutions to the inhomogeneous incompressible Navier-Stokes-Vlasov equations when the initial density is bounded away from zero.
   About hydrodynamic limits results for the other related models, cf. \cite{Aceves-Sanchez2019,Kang2020} and
   so on.

   However, the hydrodynamic limit of weak solutions
   to the compressible Navier-Stokes-Vlasov-Poisson equations remains unstudied so far,  this paper will focus on this problem.

  \bigbreak
   {\bf{Notations.}} We would like to introduce some notations for later use:
   We  denote by $C$ a generic positive constant which may
   differ from line to line.
   We denote $L^{p}(\mathbb{T}^{3})(1\leq p \leq \infty)$ the $L^p$ spaces with norm
   $\|\cdot\|_{L^{p}(\mathbb{T}^{3})}$. When there is no confusion, we will write 	·	$\|\cdot\|_{H^{k}(\mathbb{T}^{3})}$, 	 $\|\cdot\|_{L^{p}(\mathbb{T}^{3})}$	 as $\|\cdot\|_{H_{x}^{k}}$, 	$\|\cdot\|_{L_{x}^{p}}$,  respectively.
And we also write	$\|\cdot\|_{L^{p}((0,T)\times\mathbb{T}^{3})}$ 	 as 	$\|\cdot\|_{L_{t,x}^{p}}$ $(1\leq p\leq \infty)$.

   \bigbreak

   We firstly give the definition of weak solutions to the problem \eqref{O-eq1.1}-\eqref{O-ini1}.
\begin{defn}\label{O-Def3.1}
	We say $\left(f^{\epsilon}(t,x,\xi),\rho^{\epsilon}(t,x),u^{\epsilon}(t,x)\right)$ is a global weak solution of the problem \eqref{O-eq1.1}-\eqref{O-ini1}, if for any $T>0$, the following properties hold:
	\begin{itemize}
		\item $f^{\epsilon}(t,x,\xi)\geq0$,\ for any $(t,x,\xi)\in ((0,T)\times\mathbb{T}^{3}\times \mathbb{R}^{3})$;
		\item $f^{\epsilon}\in L^{\infty}(0,T;L^{1}\cap L^{\infty}(\mathbb{T}^{3}\times \mathbb{R}^{3}))$;
		\item $|\xi|^{2}f^{\epsilon}\in L^{\infty}(0,T;L^{1}(\mathbb{T}^{3}\times \mathbb{R}^{3}))$;
		\item $\Phi^{\epsilon}\in L^{\infty}(0,T;H^{1}(\mathbb{T}^{3}));$
		\item $0\leq\rho^{\epsilon}\in L^{\infty}(0,T;L^{1}(\mathbb{T}^{3}))\cap C([0,T];L^{1}(\mathbb{T}^{3}))$;
		\item $ \sqrt{\rho^{\epsilon}}u^{\epsilon}\in L^{\infty}(0,T;L^{2}(\mathbb{T}^{3}));
		u^{\epsilon} \in  L^{2}(0,T;H^{1}(\mathbb{T}^{3}))$; $\rho^{\epsilon} u^{\epsilon}\in C([0,T];L_{w}^{\frac{2\gamma}{\gamma+1}}(\mathbb{T}^{3}))$;
		
		\item For any $\varphi\in C_{c}^{\infty}([0,T]\times\mathbb{T}^{3}\times \mathbb{R}^{3})$ with $\varphi(T,\cdot,\cdot)=0$,  one has
		\begin{align}\label{O-test1}
		\int_{0}^{T}\int_{\mathbb{T}^{3}\times\mathbb{R}^{3}}&f^{\epsilon}\left(\partial_{t}\varphi+\xi\cdot\nabla_{x}\varphi
		+(u^{\epsilon}-\xi-\nabla_{x}\Phi)\cdot\nabla_{\xi}\varphi+\frac{1}{\epsilon}(u^{\epsilon}_{f}-\xi)\cdot\nabla_{\xi}\varphi\right)~d\xi dxds\nonumber\\
		&+\int_{\mathbb{T}^{3}\times\mathbb{R}^{3}}f^{\epsilon}_{0}\varphi(0,x,\xi)~d\xi dx=0.
		\end{align}
		
	\item For any $\phi\in H^{1}(\mathbb{T}^{3})$, one has
	\begin{equation}\label{O-test2}
	\int_{\mathbb{T}^{3}}\nabla_{x}\Phi^{\epsilon}\cdot\nabla_{x}\phi~dx=\int_{\mathbb{T}^{3}}(\rho^{\epsilon}_{f}-1)\phi~dx.
	\end{equation}
	
	\item For any $\psi\in C_{c}^{\infty}([0,T]\times\mathbb{T}^{3})$ with $\psi(T,\cdot)=0$, we obtain
	\begin{align}\label{O-test3}
	&\int_{0}^{T}\int_{\mathbb{T}^{3}}\rho^{\epsilon}(\partial_{t}\psi+u^{\epsilon}\cdot \nabla_{x}\psi)~dxdt+\int_{\mathbb{T}^{3}}\rho^{\epsilon}_{0}\psi(0,x)~dx=0.\\
	\int_{0}^{T}\int_{\mathbb{T}^{3}}& \Bigl\{\rho^{\epsilon} u^{\epsilon}\cdot \partial_{t}\psi+\rho^{\epsilon} u^{\epsilon}\otimes u^{\epsilon}:\nabla \psi+p^{\epsilon}\mathrm{div}_{x}\psi-\nabla u^{\epsilon}:\nabla \psi-\int_{\mathbb{R}^{3}}(u^{\epsilon}-\xi)f^{\epsilon}~d\xi\cdot \psi \Bigr\}~dxdt\nonumber\\
	&+\int_{\mathbb{T}^{3}}\rho^{\epsilon}_{0}u^{\epsilon}_{0}\cdot\psi(0,x)~dx=0.
	\end{align}
		\item The weak solution $\left(f^{\epsilon},\rho^{\epsilon},u^{\epsilon}\right)$ satisfies the  following energy inequality:
	\begin{align}\label{O-energy}
	\int_{\mathbb{T}^{3}}&\left\lbrace \rho^{\epsilon}\frac{|u^{\epsilon}|^{2}}{2}+\frac{1}{\gamma-1}(\rho^{\epsilon})^{\gamma}+\frac{1}{2}|\nabla_{x}\Phi^{\epsilon}|^{2}+\int_{\mathbb{R}^{3}}\frac{|\xi|^{2}}{2}f^{\epsilon}~d\xi\right\rbrace dx\nonumber\\ &+\int_{0}^{t}\int_{\mathbb{T}^{3}}\left\lbrace |\nabla_{x}u^{\epsilon}|^{2}+\int_{\mathbb{R}^{3}}(|u^{\epsilon}-\xi|^{2}+\frac{1}{\epsilon}|u^{\epsilon}_{f}-\xi|^{2})f~d\xi\right\rbrace ~dxds\nonumber\\
	\leq &\int_{\mathbb{T}^{3}}\left\lbrace \rho^{\epsilon}_{0}\frac{|u^{\epsilon}_{0}|^{2}}{2}+\frac{1}{\gamma-1}(\rho^{\epsilon}_{0})^{\gamma}+\frac{1}{2}|\nabla_{x}\Phi^{\epsilon}_{0}|^{2}+\int_{\mathbb{R}^{3}}\frac{|\xi|^{2}}{2}f^{\epsilon}_{0}~d\xi\right\rbrace dx,
	\end{align}
		for any $t\in [0,T].$
	\end{itemize}
\end{defn}

 Before giving the main theorem, we first make the following assumptions about the initial data:
\begin{equation}\label{O-assump-11}
\mathcal{F}(f_{0}^{\epsilon},\rho_{0}^{\epsilon},u_{0}^{\epsilon})<\infty,
\end{equation}
\begin{equation}\label{O-assump-1}
\int_{\mathbb{T}^{3}}\mathcal{H}({U}^{\epsilon}_{0}|U_{0})~dx\leq C\sqrt{\epsilon},
\end{equation}
and
\begin{equation}\label{O-assump-2}
\mathcal{F}(f_{0}^{\epsilon},\rho_{0}^{\epsilon},u_{0}^{\epsilon})-\frac{1}{2}\int_{\mathbb{T}^{3}}|\nabla K*(\rho_{f_{0}}^{\epsilon}-1)|^{2}~dx-\int_{\mathbb{T}^{3}}H(U_{0}^{\epsilon})~dx\leq C\sqrt{\epsilon},
\end{equation}
where $\mathcal{F}(f^{\epsilon},\rho^{\epsilon},u^{\epsilon})$  is defined by \eqref{O-entropy-1} and the relative entropy functional $\mathcal{H}({U}^{\epsilon}|U)$ is given by \eqref{O-relative entropy functional}.

 Now, we are in a position to state our main result.
\begin{thm}\label{O-D THM1}
Let $T>0$ and $\gamma>\frac{3}{2}$,  and let $(f^{\epsilon},\rho^{\epsilon},u^{\epsilon})$ be a global weak
solution to the problem $\eqref{O-eq1.1}-\eqref{O-ini1}$ with the initial data $(f^{\epsilon}_0 , \rho^{\epsilon}_0, u^{\epsilon}_0)$ in the sense of Definition
\ref{O-Def3.1}, and $(\rho,u,\rho_{f},u_{f})$ be a smooth solution of the limit system \eqref{O-limit system} with the smooth initial data
$(\rho_{0},u_{0},\rho_{f_{0}},u_{f_{0}})$, which satisfies the assumption conditions $\eqref{O-assump-11}-\eqref{O-assump-2}$.
Then there exists a positive constant $C$  such that
	\begin{align}\label{O-enery 2}
	\int_{\mathbb{T}^{3}}&\mathcal{H}(U^{\epsilon}|U)~dx+\int_{\mathbb{T}^{3}}|\nabla K*({\rho}^{\epsilon}_{f}-\rho_{f})|^{2}~dx\nonumber\\
	 &+\int_{0}^{T}\int_{\mathbb{T}^{3}}{\rho}^{\epsilon}_{f}|({u}^{\epsilon}_{f}-{u}^{\epsilon})-({u}_{f}-{u})|^{2}~dxds+\int_{0}^{T}\int_{\mathbb{T}^{3}}|\nabla(u-{u}^{\epsilon})|^{2}~dxds\nonumber\\
	\leq& C\sqrt{\epsilon}.
	\end{align}
Besides, the following convergence
	hold:
	$$\rho^{\epsilon}\longrightarrow \rho \ \ \ \mathrm{in} \ \ \  L^{1}_{loc}(0,T;L^{p}(\mathbb{T}^{3})), p\in[1,\gamma];$$
	$$\rho_{f}^{\epsilon}\longrightarrow \rho_{f} \ \ \ \mathrm{in} \ \ \  L^{2}(0,T;H^{-1}(\mathbb{T}^{3}));$$
		$${f}^{\epsilon}\longrightarrow \rho_{f}\otimes\delta_{\xi=u_{f}} \ \ \ \mathrm{in} \ \ \  \mathcal{D}^{'}((0,T)\times\mathbb{T}^{3}\times\mathbb{R}^{3});$$
		$$\rho^{\epsilon}u^{\epsilon}\longrightarrow \rho u \ \ \ \mathrm{in} \ \ \  L^{\infty}(0,T;L^{1}(\mathbb{T}^{3}));$$
	$$\rho_{f}^{\epsilon}u_{f}^{\epsilon}\longrightarrow \rho_{f}u_{f} \ \ \ \mathrm{in} \ \ \  \mathcal{D}^{'}((0,T)\times\mathbb{T}^{3}).$$
\end{thm}

\begin{rem}
The constant $C>0$ depends on $\|\rho_{f}\|_{L_{t,x}^{\infty}}$, $\rho_{*}(:=\inf\limits_{(t,x)\in[0,T]\times\mathbb{T}^{3}} \rho(t,x)>0)$, $\bar{C}$, $\gamma$, $\|\rho\|_{L_{t,x}^{\infty}}$, $\|u_{f}-u\|_{L_{t,x}^{\infty}}$, $\|\nabla(u_{f}-u)\|_{L_{t,x}^{\infty}}$, $\|\Delta u\|_{L_{t,x}^{\infty}}$,  $\|\nabla u_{f}\|_{L_{t,x}^{\infty}}$, where $\bar{C}$ is defined in Lemma \ref{O-lem-ineq}.
\end{rem}
\begin{rem}
	This result also holds for compressible Navier-Stokes-Vlasov-Poisson equations without local alignment force. But we can not obtain the hydrodynamic limit of weak solutions to the  3D compressible Navier-Stokes-Vlasov equations, because of without  good convergence of $\rho_{f}^{\epsilon}$ and  estimate of  $\mathcal{J}_{5}^{2}$ (see \eqref{O-J52})when the particles are not charged to the fluid-particle model. At present, we only obtain  the hydrodynamic limit of weak solutions to  the  one-dimensional compressible Navier-Stokes-Vlasov system with the local alignment force, see \cite{Shi2023}.  We will consider the hydrodynamic limit of weak solutions to the  multidimensional compressible Navier-Stokes-Vlasov equations in the future.
\end{rem}
\begin{rem}
The global weak solutions for the compressible Navier-Stokes-Vlasov-Poisson equations with local alignment force  is reported in our preprint paper  \cite{Su2025}.
\end{rem}

Now, let us sketch the strategy of proving Theorem \ref{O-D THM1} and explain some main difficulties and techniques  arguments involved in the process.
As  mentioned in the introduction, Choi and Jung \cite{Choi2021} proved  the global weak solutions of the compressible Navier-Stokes-Vlasov-Fokker-Planck equations with local alignment force converge to the smooth solutions of the limiting two-phase fluid model,
they defined the relative entropy functional as follows:
$$\bar{\mathcal{H}}(U^{\epsilon}|U)=\frac{1}{2}{\rho}^{\epsilon}|{u}^{\epsilon}-u|^{2}+\frac{1}{2}{\rho}^{\epsilon}_{f}|{u}^{\epsilon}_{f}-u_{f}|^{2}+P({\rho}^{\epsilon}|\rho)+\bar{P}({\rho^{\epsilon}_{f}}|\rho_{f}),$$
where
$$P({\rho}^{\epsilon}|\rho)=\frac{1}{\gamma-1}(({\rho}^{\epsilon})^{\gamma}-{\rho}^{\gamma})+\frac{\gamma}{\gamma-1}(\rho-{\rho}^{\epsilon})\rho^{\gamma-1},$$
and
$$\bar{P}({\rho^{\epsilon}_{f}}|\rho_{f})=\rho^{\epsilon}_{f}\log \rho^{\epsilon}_{f}-\rho_{f}\log \rho_{f}+(\rho_{f}-\rho^{\epsilon}_{f})(1+\log \rho_{f}).$$ But when there is no the diffusion term $\Delta f^{\epsilon}$ in the kinetic equation,  the relative entropy functional does not include $\bar{P}(\rho_{f}^{\epsilon}|\rho_{f})$. The difficulty  encountered  is that we can not obtain the convergence of $\rho_{f}^{\epsilon}$ from the relative entropy in our analysis. And then we can not deal with the term  $\int_{0}^{t}\int_{\mathbb{T}^{3}}(\rho_{f}^{\epsilon }-\rho_{f})(u_{f}-u)\cdot({u}^{\epsilon}-u)~dxds$. Han-Kwan and Michel \cite{Han-Kwan2024} established
the hydrodynamic limits for the incompressible Navier-Stokes-Vlasov equations
on $\mathbb{T}^{3}$. And later, based on some ideas in  \cite{Han-Kwan2024},
Su-Wu-Yao-Zhang \cite{Su2023} established the hydrodynamic limit of the inhomogeneous incompressible Navier-Stokes-Vlasov equations when the initial density is bounded away from zero. Precisely,
assume {\it a priori} $\|\nabla u^\epsilon\|_{L^1(0,T;L^\infty(\mathbb{T}^3))}\leq \delta $, which implies $\|\rho_{f}^\epsilon\|_{L^\infty((0,T)\times\mathbb{T}^3)}\leq C$, the authors obtained
\begin{align}\label{O-est-eq47-n}
&\|\rho_{f}^{\epsilon}-\rho_{f}\|_{\dot{H}^{-1}(\mathbb{T}^3)}\leq \|\rho^{\epsilon}_{f_{0}}-\rho_{f_{0}}\|_{\dot{H}^{-1}(\mathbb{T}^3)}+\nonumber\\
&\epsilon\int_{0}^{t}\|F^\epsilon(s)\|_{L^2(\mathbb{T}^3)} ds+ C(\|\rho_{f}^\epsilon\|_{L^{\infty}((0,T)\times\mathbb{T}^{3})})\int_{0}^{t}\|u^{\epsilon}-u\|_{L^{2}(\mathbb{T}^3)}~ds,
\end{align}
then they can estimate
$\int_{0}^{t}\int_{\mathbb{T}^{3}}(\rho_{f}^{\epsilon}-\rho_{f})(u^{\epsilon}-u)\cdot G~dxds$ in \cite{Su2023}, here $F^\epsilon=\frac{1}{\epsilon}\int_{\mathbb{R}^3}(\xi-u^\epsilon)f^\epsilon d\xi$, and $G$ is smooth related to the  physical quantity of  limit equations.  At last, they used maximal parabolic regularity estimates for the incompressible Navier-Stokes equations and regularity estimate for the Stokes equations to close the  {\it a priori}  assumption $\|\nabla u^\epsilon\|_{L^1(0,T;L^\infty(\mathbb{T}^3))}\leq \delta$. But this method does not apply for the compressible case.
Recently, Choi-Jung \cite{Choi2023} considered the hydrodynamic limit of the Vlasov-Poisson or Vlasov-Poisson-Fokker-Planck equation coupled with the incompressible Navier-Stokes equations,
 it is surprising that the Coulomb interaction can be used to control $\rho_{f}^{\epsilon}-\rho_{f}$,  due to $-\Delta K*(\rho_{f}^{\epsilon}-\rho_{f})=\rho_{f}^{\epsilon}-\rho_{f}$ in the weak sense.
 Motivated by this,  we can estimate  $\int_{0}^{t}\int_{\mathbb{T}^{3}}(\rho_{f}^{\epsilon }-\rho_{f})(u_{f}-u)\cdot({u}^{\epsilon}-u)~dxds$.
Thus this paper considers the hydrodynamic limit of weak solutions to the compressible Navier-Stokes-Vlasov-Poisson equations with the local alignment force.
We firstly obtain the relative entropy inequality in
Lemma \ref{O-relative ineq} to the problem  $\eqref{O-eq1.1}-\eqref{O-ini1}$, and then show the relative entropy
estimates (see Section \ref{O-Section2-3}), and finally pass to the limit of $\epsilon\rightarrow 0$. In the process of estimating the relative entropy
inequality, $\mathcal{J}_1$ can be solved by assumption \eqref{O-assump-1}. $\mathcal{J}_2$ and $\mathcal{J}_3$ can be estimated by $\mathcal{O}(\sqrt{\epsilon})$, where we used the fact that $\rho^{\epsilon}_{f}|u^{\epsilon}_{f}|^{2}\leq\int_{\mathbb{R}^{3}}|\xi|^{2}f^{\epsilon}~d\xi$, assumption \eqref{O-assump-2} and Lemma
 \ref{O-modified entropy estimate}. $\mathcal{J}_4$ and $\mathcal{J}_6$ can be controlled by $C\int_{\mathbb{T}^{3}}\mathcal{H}(U^{\epsilon}|U)~dx$. For $\mathcal{J}_5$,
  due to $-\Delta K*(\rho_{f}^{\epsilon}-\rho_{f})=\rho_{f}^{\epsilon}-\rho_{f}$ in the weak sense and the Poincar\'{e} type inequality \eqref{O-lem poincare}, we can estimate
 \begin{align*}
 \mathcal{J}_{5}^{2}=\int_{0}^{t}\int_{\mathbb{T}^{3}}(\rho_{f}^{\epsilon }-\rho_{f})(u_{f}-u)\cdot({u}^{\epsilon}-u)~dxds.
 \end{align*}
 For $\mathcal{J}_7$, we use the symmetry of $K$ to obtain
 \begin{align*}
 \mathcal{J}_7+\frac{1}{2}\int_{\mathbb{T}^{3}}|\nabla K*(\rho_{f}^{\epsilon}-\rho_{f})|^{2}~dx\leq \frac{1}{2}\int_{\mathbb{T}^{3}}|\nabla K*(\rho_{f_{0}}^{\epsilon}-\rho_{f_{0}})|^{2}~dx+C\int_{0}^{t}\int_{\mathbb{T}^{3}}|\nabla K*(\rho_{f}^{\epsilon}-\rho_{f})|^{2}~dxds.
 \end{align*}
Finally, we use the  Gr\"{o}nwall lemma to get the desired relative entropy inequality \eqref{O-enery 2}.

\bigbreak
The rest of this paper is organized as follows. In Section \ref{O-Section-1}, we give the formal derivation of the asymptotic system.  In Section  \ref{O-Section2}, we give the proof of Theorem \ref{O-D THM1}.
In subsection \ref{O-Section 2-2}, we give the relative entropy inequality of the problem \eqref{O-eq1.1}-\eqref{O-ini1};
Subsection \ref{O-Section2-3} is devoted to obtain the relative entropy estimate in order to proof Theorem  \ref{O-D THM1}. In
subsection \ref{O-Section2-4}, by passing to the limit as $\epsilon\rightarrow 0$, we complete the proof of Theorem \ref{O-D THM1}.

\section{Formal derivation of the asymptotic model}\label{O-Section-1}
Denote
\begin{equation*}
\rho_{f}^{\epsilon}(t,x)=\int_{\mathbb{R}^{3}} f^{\epsilon}(t,x,\xi)~d\xi, ~~~~~~u_{f}^{\epsilon}(t,x)=\frac{1}{\rho_{f}^{\epsilon}(t,x)}\int_{\mathbb{R}^{3}} \xi f^{\epsilon}(t,x,\xi)~d\xi.
\end{equation*}
Firstly, integrating $\eqref{O-eq1.1}_{1}$ with respect to $\xi$, yields
\begin{equation*}
\partial_t{\rho_{f}^{\epsilon}}+\mathrm{div}_{x}{(\rho_{f}^{\epsilon}u_{f}^{\epsilon})}=0.
\end{equation*}
Next, multiplying $\eqref{O-eq1.1}_{1}$ by $\xi$ and integrating with respect to $\xi$, yields
\begin{equation*}
\partial_{t}(\rho_{f}^{\epsilon}u_{f}^{\epsilon})+\mathrm{div}_{x}(\rho_{f}^{\epsilon}u_{f}^{\epsilon}\otimes u_{f}^{\epsilon})+\mathrm{div}_{x}\left(\int_{\mathbb{R}^{3}}(\xi-u_{f}^{\epsilon})\otimes (\xi-u_{f}^{\epsilon})f^{\epsilon}~d\xi\right)=\rho_{f}^{\epsilon}(u^{\epsilon}-u_{f}^{\epsilon})-\rho_{f}^{\epsilon}\nabla_{x}\Phi^{\epsilon},
\end{equation*}
then integrating the third term on left-hand side of the above equation with respect to $x$ and $t$, yields
\begin{align}\label{O-edu-eq1.5}
\int_{0}^{T}\int_{\mathbb{T}^{3}\times\mathbb{R}^{3}}(\xi-u_{f}^{\epsilon})\otimes (\xi-u_{f}^{\epsilon})f^{\epsilon}~d\xi dxds
& \leq \left(\int_{0}^{T}\int_{\mathbb{T}^{3}\times\mathbb{R}^{3}}|\xi-u_{f}^{\epsilon}|^{2}f^{\epsilon}~d\xi dxds\right)\nonumber\\
&\leq C\epsilon,
\end{align}
where we have used the fact that
\begin{equation*}
\int_{0}^{T}\int_{\mathbb{T}^{3}\times\mathbb{R}^{3}}|\xi-u_{f}^{\epsilon}|^{2}f^{\epsilon}~d\xi dxds\leq C\epsilon,
\end{equation*}
which can be obtained by the energy inequality \eqref{O-energy}.

As a result, as $\epsilon\rightarrow0$,  \eqref{O-edu-eq1.5} converges to 0. We assume that $\rho_{f}^{\epsilon}$, $u_{f}^{\epsilon}$ converges to $\rho_{f}$, $u_{f}$ and  $\rho^{\epsilon}$, $u^{\epsilon}$ converges to $\rho$, $u$  respectively. Then we can pass to the limits of the nonlinear terms $\rho_{f}^{\epsilon}u_{f}^{\epsilon}$ and $\rho_{f}^{\epsilon} u_{f}^{\epsilon}\otimes u_{f}^{\epsilon}$, and obtain the corresponding limits $\rho_{f} u_{f}$ and $\rho_{f} u_{f}\otimes u_{f}$. And the nonlinear terms $\rho^{\epsilon}u^{\epsilon}$ and $\rho^{\epsilon} u^{\epsilon}\otimes u^{\epsilon}$ converges to the corresponding limits $\rho u$ and $\rho u\otimes u$, respectively.
So, formally, we have the following system:
\begin{equation}\label{O-limit system}
\left\{
\begin{array}{l}
\partial_t{\rho_{f}}+\mathrm{div}_{x}{(\rho_{f}u_{f})}=0,\\
\partial_{t}\rho+\mathrm{div}_{x}(\rho u)=0,\\
-\Delta_{x}\Phi=\rho_{f}-1,\\
\partial_{t}(\rho_{f}u_{f})+\mathrm{div}_{x}(\rho_{f}u_{f}\otimes u_{f})=\rho_{f}(u-u_{f})-\rho_{f}\nabla_{x}\Phi,\\
\partial_{t}(\rho u)+\mathrm{div}_{x}(\rho u\otimes u)+\nabla_{x}p-\Delta_{x} u=-\rho_{f}(u-u_{f}).
\end{array}
\right.
\end{equation}
And furthermore the limit system \eqref{O-limit system} can be rewritten as follows:
\begin{equation}\label{O-MODI-limitsystem}
\left\{
\begin{array}{l}
\partial_t{\rho_{f}}+\mathrm{div}_{x}{(\rho_{f}u_{f})}=0,\\
\partial_{t}(\rho_{f}u_{f})+\mathrm{div}_{x}(\rho_{f}u_{f}\otimes u_{f})=\rho_{f}(u-u_{f})-\rho_{f}\nabla_{x}K*(\rho_{f}-1),\\
\partial_{t}\rho+\mathrm{div}_{x}(\rho u)=0,\\
\partial_{t}(\rho u)+\mathrm{div}_{x}(\rho u\otimes u)+\nabla_{x}p-\Delta_{x} u=-\rho_{f}(u-u_{f}).
\end{array}
\right.
\end{equation}
\section{Proof of Theorem \ref{O-D THM1}}\label{O-Section2}
\subsection{Relative entropy inequality}\label{O-Section 2-2}
In this subsection, we deduce the relative entropy inequality.
Setting
\begin{equation}\label{O-entropy-1}
\mathcal{F}(f,\rho,u):=\int_{\mathbb{T}^{3}}\left\lbrace  \frac{1}{2}\rho|u|^{2}+\frac{1}{\gamma-1}\rho^{\gamma}+\frac{1}{2}|\nabla_{x}K*(\rho_{f}-1)|^{2}+\int_{\mathbb{R}^{3}}\frac{1}{2}|\xi|^{2}f~d\xi\right\rbrace dx,
\end{equation}
\begin{equation}\label{O-entropy-diffu-2}
\mathcal{D}_{1}(f):=\int_{\mathbb{T}^{3}\times\mathbb{R}^{3}}|u_{f}-\xi|^{2}f~dxd\xi,
\end{equation}
\begin{equation}\label{O-entropy-diffu-3}
\mathcal{D}_{2}(f,u):=\int_{\mathbb{T}^{3}\times\mathbb{R}^{3}}|u-\xi|^{2}f~dxd\xi+\int_{\mathbb{T}^{3}}|\nabla_{x}u|^{2}~dx.
\end{equation}

Now we give the entropy inequality to the problem \eqref{O-eq1.1}-\eqref{O-ini1}.
\begin{lem}\label{O-entropy estimate}
For $T>0$, suppose that $(f^{\epsilon},\rho^{\epsilon},u^{\epsilon})$ is a weak solution to the problem \eqref{O-eq1.1}-\eqref{O-ini1} on the interval $[0,T)$ with the initial data $(f_{0}^{\epsilon},\rho_{0}^{\epsilon},u_{0}^{\epsilon})$ in the sense of Definition \ref{O-Def3.1}. Then we have
	\begin{align*}
 \mathcal{F}(f^{\epsilon},\rho^{\epsilon},u^{\epsilon})(t)+\frac{1}{\epsilon}\int_{0}^{t}\mathcal{D}_{1}(f^{\epsilon})(s)~ds+\int_{0}^{t}\mathcal{D}_{2}(f^{\epsilon},u^{\epsilon})(s)~ds\leq \mathcal{F}(f_{0}^{\epsilon},\rho_{0}^{\epsilon},u_{0}^{\epsilon}).
	\end{align*}
\end{lem}
\begin{proof}
	For the proof, we refer to Ref.\cite{Choi2021}, here we omit it.
\end{proof}
Next, we give a modified entropy inequality for later use.
\begin{lem}\label{O-modified entropy estimate}
For $T>0$, suppose that $(f^{\epsilon},\rho^{\epsilon},u^{\epsilon})$ is a weak solution to the problem \eqref{O-eq1.1}-\eqref{O-ini1} on the interval $[0,T)$ with the initial data $(f_{0}^{\epsilon},\rho_{0}^{\epsilon},u_{0}^{\epsilon})$ in the sense of Definition \ref{O-Def3.1}. Then we have	
	\begin{align*}
	 \mathcal{F}(f^{\epsilon},\rho^{\epsilon},u^{\epsilon})(t)&+\frac{1}{2\epsilon}\int_{0}^{t}\mathcal{D}_{1}(f^{\epsilon})(s)~ds+\int_{0}^{t}\int_{\mathbb{T}^{3}}\rho_{f}^{\epsilon}|u_{f}^{\epsilon}-u|^{2}~dxds+\int_{0}^{t}\int_{\mathbb{T}^{3}}|\nabla_{x}u^{\epsilon}|^{2}~dxds\\
	&\leq \mathcal{F}(f_{0}^{\epsilon},\rho_{0}^{\epsilon},u_{0}^{\epsilon})+C(T)\epsilon.
	\end{align*}
\end{lem}
\begin{proof}
For the proof, we refer to Ref.\cite{Choi2021}, here we omit it.
\end{proof}
\begin{lem}\label{O-Lem-relative}
(\cite{Choi2021}.) Let $x,y>0$ and $\gamma>1$. If $0<y_{\min}\leq y\leq y_{\max}<\infty$, then the following inequality holds:
\begin{align*}
P(x|y)&=\frac{1}{\gamma-1}(x^{\gamma}-y^{\gamma})+\frac{\gamma}{\gamma-1}(y-x)y^{\gamma-1}\\
&\geq C \left\{
\begin{array}{l}
(x-y)^2,\ \ if \ \ \frac{y}{2}\leq x\leq 2y,\\
(1+x^{\gamma}),\ \  otherwise,
\end{array}
\right.
\end{align*}
where $C=C(\gamma, y_{min}, y_{max})$ is a positive constant.
\end{lem}

The following Poincar\'{e} type inequality can be found in \cite{Feireisl2004}.
\begin{lem}\label{O-lem-ineq}
(\cite{Feireisl2004}.) Let $u\in H^{1}(\mathbb{T}^{3})$, and let $\rho$ be a non-negative function satisfying
\begin{equation*}
0<C_{1}\leq \int_{\mathbb{T}^{3}} \rho~dx,\ \int_{\mathbb{T}^{3}} \rho^\gamma~dx \leq C_{2}.
\end{equation*}
 Then there exists a positive constant $\bar{C}$ depending only on $C_{1}$, $C_{2}$ and $\gamma$ such that
\begin{equation}\label{O-lem poincare}
\|u\|_{L^{2}(\mathbb{T}^{3})}^{2}\leq \bar{C}\int_{\mathbb{T}^{3}} \rho|u|^{2}~dx+\bar{C}\|\nabla u\|_{L^{2}(\mathbb{T}^{3})}^{2}.
\end{equation}
\end{lem}

 In order to deduce the relative entropy inequality, we define
\begin{center}
$U$:=$\begin{pmatrix}
\rho_{f}\\
\omega\\
\rho\\
m
\end{pmatrix}$,\ \ \ \ \
$A(U)$:=$\begin{pmatrix}
\omega\\
\frac{\omega\otimes\omega}{\rho_{f}}\\
m\\
\frac{m\otimes m}{\rho}+\rho^{\gamma }I_{3\times3}
\end{pmatrix}$,\\
\end{center}
and
\begin{center}
	$F(U)$:=$\begin{pmatrix}
	0\\
	\rho_{f}(u-u_{f})-\rho_{f}\nabla_{x}K*(\rho_{f}-1)\\
	0\\
	-\rho_{f}(u-u_{f})+\Delta_{x}u
	\end{pmatrix}$.
\end{center}
where $\omega:=\rho_{f}u_{f}$ and $m:=\rho u$. Note that the system \eqref{O-MODI-limitsystem} can be written in the form of conservation of laws:
$$\partial_{t}U+\mathrm{div}_{x}A(U)=F(U).$$
Then the associated energy $H(U)$ to the above system can be written as
$$H(U)=\frac{\omega^{2}}{2\rho_{f}}+\frac{m^{2}}{2\rho}+\frac{1}{\gamma-1}\rho^{\gamma}.$$
We now introduce the relative entropy functional $\mathcal{H}$ as follows:
$$\mathcal{H}(\bar{U}|U):=H(\bar{U})-H(U)-DH(U)(\bar{U}-U),$$
where
\begin{center}
	$\bar{U}$:=$\begin{pmatrix}
	\bar{\rho}_{f}\\
	\bar{\omega}\\
	\bar{\rho}\\
	\bar{m}
	\end{pmatrix}$,
\end{center}
by the direct calculation, we obtain
\begin{equation}\label{O-relative entropy functional}
	\mathcal{H}(\bar{U}|U)=\frac{1}{2}\bar{\rho}_{f}|\bar{u}_{f}-u_{f}|^{2}+\frac{1}{2}\bar{\rho}|\bar{u}-u|^{2}+P(\bar{\rho}|\rho),
\end{equation}
where
\begin{align}\label{O-relative-pressure}
	P(\bar{\rho}|\rho)=\frac{1}{\gamma-1}(\bar{\rho}^{\gamma}-{\rho}^{\gamma})+\frac{\gamma}{\gamma-1}(\rho-\bar{\rho})\rho^{\gamma-1}\geq \gamma \min\left\lbrace \bar{\rho}^{\gamma-2},\rho^{\gamma-2}\right\rbrace (\bar{\rho}-\rho)^{2}.
\end{align}

\begin{lem}\label{O-relative ineq}
The relative entropy functional $\mathcal{H}(\bar{U}|U)$ defined by $\eqref{O-relative entropy functional}$ satisfies
\begin{align*}
&\int_{\mathbb{T}^{3}}\mathcal{H}(\bar{U}|U)~dx+\int_{0}^{t}\int_{\mathbb{T}^{3}}\bar{\rho}_{f}|(\bar{u}_{f}-\bar{u})-({u}_{f}-{u})|^{2}~dxds+\int_{0}^{t}\int_{\mathbb{T}^{3}}|\nabla(u-\bar{u})|^{2}~dxds\\
=&\int_{\mathbb{T}^{3}}\mathcal{H}(\bar{U}_{0}|U_{0})~dx+\int_{0}^{t}\int_{\mathbb{T}^{3}}\partial_{s} H(\bar{U})~dxds+\int_{0}^{t}\int_{\mathbb{T}^{3}}|\nabla \bar{u}|^{2}~dxds+\int_{0}^{t}\int_{\mathbb{T}^{3}}\bar{\rho}_{f}|\bar{u}_{f}-u|^{2}~dxds\\
&+\int_{0}^{t}\int_{\mathbb{T}^{3}}\bar{\rho}_{f}\bar{u}_{f}\cdot\nabla K*(\bar{\rho}_{f}-1)~dxds-\int_{0}^{t}\int_{\mathbb{T}^{3}}DH(U)(\partial_{s}\bar{U}+\mathrm{div}_{x}A(\bar{U})-F(\bar{U}))~dxds\\
&-\int_{0}^{t}\int_{\mathbb{T}^{3}}\nabla DH(U):A(\bar{U}|U)~dxds+\int_{0}^{t}\int_{\mathbb{T}^{3}}(\frac{\bar{\rho}}{\rho}\rho_{f}-\bar{\rho}_{f})(u_{f}-u)\cdot(u-\bar{u})~dxds\\
&+\int_{0}^{t}\int_{\mathbb{T}^{3}}(\frac{\bar{\rho}}{\rho}-1)\Delta u\cdot(u-\bar{u})~dxds-\int_{0}^{t}\int_{\mathbb{T}^{3}}\bar{\rho}_{f}(u_{f}-\bar{u}_{f})\cdot\nabla K*(\rho_{f}-\bar{\rho}_{f})~dxds.
\end{align*}
\end{lem}

\begin{proof}
A direct calculations yields
\begin{align*}
\frac{d}{dt}\int_{\mathbb{T}^{3}}\mathcal{H}(\bar{U}|U)~dx=&\int_{\mathbb{T}^{3}}\left\lbrace\partial_{t} H(\bar{U})-\partial_{t} H(U)-D^{2}H(U)\partial_{t}U(\bar{U}-U)-DH(U)(\partial_{t}\bar{U}-\partial_{t}{U})\right\rbrace~dx\\
=&\int_{\mathbb{T}^{3}}\partial_{t} H(\bar{U})~dx-\int_{\mathbb{T}^{3}}DH(U)(\partial_{t}\bar{U}+\mathrm{div}_{x}A(\bar{U})-F(\bar{U}))~dx\\
&+\int_{\mathbb{T}^{3}}D^{2}H(U)\mathrm{div}_{x}A(U)(\bar{U}-U)+DH(U)\mathrm{div}_{x}A(\bar{U})~dx\\
&-\int_{\mathbb{T}^{3}}\left\lbrace D^{2}H(U)F(U)(\bar{U}-U)+DH(U)F(\bar{U})\right\rbrace~dx\\
=&\sum_{i=1}^{4}I_{i}.
\end{align*}
For $I_{3}$. By using the same argument in \cite{Choi2021}, we have
\begin{equation}\label{O-I3}
I_{3}=-\int_{\mathbb{T}^{3}}\nabla DH(U):A(\bar{U}|U)~dx.
\end{equation}
For $I_{4}$, due to
\begin{center}
	$DH(U)$=$\begin{pmatrix}
	-\frac{\omega^{2}}{2\rho_{f}^{2}}\\
	\frac{\omega}{\rho_{f}}\\
	-\frac{m^{2}}{2\rho^{2}}+\frac{\gamma}{\gamma-1}\rho^{\gamma-1}\\
	\frac{m}{\rho}
	\end{pmatrix}$,
\end{center}
and
\begin{center}
	$D^{2}H(U)$=$\begin{pmatrix}
	\frac{\omega^{2}}{\rho_{f}^{3}}& -\frac{\omega}{\rho_{f}^{2}}& 0& 0\\
	-\frac{\omega}{\rho_{f}^{2}}& \frac{1}{\rho_{f}}& 0&  0 \\
	0&  0&\frac{m^{3}}{\rho^{3}}+\gamma\rho^{\gamma-2}& -\frac{m}{\rho^{2}}\\
	0&  0&  -\frac{m}{\rho^{2}}& \frac{1}{\rho}
	\end{pmatrix}$,
\end{center}
then we obtain
\begin{align*}
D^{2}H(U)F(U)(\bar{U}-U)=&-u_{f}\cdot(u-u_{f})(\bar{\rho}_{f}-\rho_{f})+u_{f}\cdot\nabla K*(\rho_{f}-1)(\bar{\rho}_{f}-\rho_{f})\\
&+(u-u_{f})\cdot(\bar{\rho}_{f}\bar{u}_{f}-\rho_{f}u_{f})-\nabla K*(\rho_{f}-1)\cdot(\bar{\rho}_{f}\bar{u}_{f}-\rho_{f}u_{f})\\
&+\frac{\rho_{f}}{\rho}u\cdot(u-u_{f})(\bar{\rho}-\rho)-\frac{u\cdot\Delta u}{\rho}(\bar{\rho}-\rho)-\frac{\rho_{f}}{\rho}(u-u_{f})\cdot(\bar{\rho}\bar{u}-\rho u)\\
&+\frac{1}{\rho}\Delta u\cdot(\bar{\rho}\bar{u}-\rho u),
\end{align*}
and
\begin{align*}
DH(U)F(U)=\bar{\rho}_{f}u_{f}\cdot\bar{u}-\bar{\rho}_{f}\bar{u}_{f}\cdot{u}_{f}-\bar{\rho}_{f}u_{f}\cdot\nabla K*(\bar{\rho}_{f}-1)-\bar{\rho}_{f}u\cdot\bar{u}+\bar{\rho}_{f}\bar{u}_{f}\cdot{u}+u\cdot\Delta \bar{u}.
\end{align*}
Thus, one has
\begin{align*}
D^{2}H(U)&F(U)(\bar{U}-U)+DH(U)F(\bar{U})\\
=&\bar{\rho}_{f}|(\bar{u}_{f}-\bar{u})-({u}_{f}-{u})|^{2}+(\frac{\bar{\rho}}{\rho}\rho_{f}-\bar{\rho}_{f})(u-u_{f})\cdot(u-\bar{u})\\
&-\bar{\rho}_{f}|\bar{u}_{f}-u|^{2}+(\frac{\bar{\rho}}{\rho}-1)\Delta u\cdot(\bar{u}-u)-(u-\bar{u})\cdot\Delta(u-\bar{u})\\
&+\bar{u}\cdot\Delta \bar{u}+\bar{\rho}_{f}(u_{f}-\bar{u}_{f})\cdot\nabla K*(\rho_{f}-\bar{\rho}_{f})-\bar{\rho}_{f}\bar{u}_{f}\cdot\nabla K*(\bar{\rho}_{f}-1),
\end{align*}
then
\begin{align*}
I_{4}=&-\int_{\mathbb{T}^{3}}D^{2}H(U)F(U)(\bar{U}-U)+DH(U)F(\bar{U})~dx\\
=&-\int_{\mathbb{T}^{3}}\bar{\rho}_{f}|(\bar{u}_{f}-\bar{u})-({u}_{f}-{u})|^{2}~dx-\int_{\mathbb{T}^{3}}(\frac{\bar{\rho}}{\rho}\rho_{f}+\bar{\rho}_{f})(u_{f}-u)\cdot(u-\bar{u})~dx\\
&+\int_{\mathbb{T}^{3}}\bar{\rho}_{f}|\bar{u}_{f}-u|^{2}~dx-\int_{\mathbb{T}^{3}}(\frac{\bar{\rho}}{\rho}-1)\Delta u\cdot(\bar{u}-u)~dx-\int_{\mathbb{T}^{3}}|\nabla(u-\bar{u})|^{2}~dx\\
&+\int_{\mathbb{T}^{3}}|\nabla \bar{u}|^{2}~dx-\int_{\mathbb{T}^{3}}\bar{\rho}_{f}(u_{f}-\bar{u}_{f})\cdot\nabla K*(\rho_{f}-\bar{\rho}_{f})~dx+\int_{\mathbb{T}^{3}}\bar{\rho}_{f}\bar{u}_{f}\cdot\nabla K*(\bar{\rho}_{f}-1)~dx.
\end{align*}
Therefore,
\begin{align*}
\frac{d}{dt}&\int_{\mathbb{T}^{3}}\mathcal{H}(\bar{U}|U)~dx+\int_{\mathbb{T}^{3}}\bar{\rho}_{f}|(\bar{u}_{f}-\bar{u})-({u}_{f}-{u})|^{2}~dx+\int_{\mathbb{T}^{3}}|\nabla(u-\bar{u})|^{2}~dx\\
=&\int_{\mathbb{T}^{3}}\partial_{t} H(\bar{U})~dx+\int_{\mathbb{T}^{3}}|\nabla \bar{u}|^{2}~dx+\int_{\mathbb{T}^{3}}\bar{\rho}_{f}|\bar{u}_{f}-u|^{2}~dx+\int_{\mathbb{T}^{3}}\bar{\rho}_{f}\bar{u}_{f}\cdot\nabla K*(\bar{\rho}_{f}-1)~dx\\
&-\int_{\mathbb{T}^{3}}DH(U)(\partial_{t}\bar{U}+\mathrm{div}_{x}A(\bar{U})-F(\bar{U}))~dx-\int_{\mathbb{T}^{3}}\nabla DH(U):A(\bar{U}|U)~dx\\
&+\int_{\mathbb{T}^{3}}(\frac{\bar{\rho}}{\rho}\rho_{f}-\bar{\rho}_{f})(u_{f}-u)\cdot(u-\bar{u})~dx+\int_{\mathbb{T}^{3}}(\frac{\bar{\rho}}{\rho}-1)\Delta u\cdot(u-\bar{u})~dx\\
&-\int_{\mathbb{T}^{3}}\bar{\rho}_{f}(u_{f}-\bar{u}_{f})\cdot\nabla K*(\rho_{f}-\bar{\rho}_{f})~dx.
\end{align*}
\end{proof}
\subsection{Relative entropy estimates}\label{O-Section2-3}
Now we replace $\bar{U}$ with $U^{\epsilon}$ in Lemma \ref{O-relative ineq}, one has
\begin{align*}
&\int_{\mathbb{T}^{3}}\mathcal{H}(U^{\epsilon}|U)~dx+\int_{0}^{t}\int_{\mathbb{T}^{3}}{\rho}^{\epsilon}_{f}|({u}^{\epsilon}_{f}-{u}^{\epsilon})-({u}_{f}-{u})|^{2}~dxds+\int_{0}^{t}\int_{\mathbb{T}^{3}}|\nabla(u-{u}^{\epsilon})|^{2}~dxds\\
=&\underbrace{\int_{\mathbb{T}^{3}}\mathcal{H}({U}^{\epsilon}_{0}|U_{0})~dx}_{\mathcal{J}_1}\underbrace{+\int_{0}^{t}\int_{\mathbb{T}^{3}}\partial_{s} H(U^{\epsilon})~dxds+\int_{0}^{t}\int_{\mathbb{T}^{3}}|\nabla {u}^{\epsilon}|^{2}~dxds+\int_{0}^{t}\int_{\mathbb{T}^{3}}{\rho}^{\epsilon}_{f}|{u}^{\epsilon}_{f}-u|^{2}~dxds}_{\mathcal{J}_2}\\
&\underbrace{+\int_{0}^{t}\int_{\mathbb{T}^{3}}{\rho}^{\epsilon}_{f}{u}^{\epsilon}_{f}\cdot\nabla K*({\rho}^{\epsilon}_{f}-1)~dxds}_{\mathcal{J}_2}\underbrace{-\int_{0}^{t}\int_{\mathbb{T}^{3}}DH(U)(\partial_{s}{U}^{\epsilon}+\mathrm{div}_{x}A({U}^{\epsilon})-F({U}^{\epsilon}))~dxds}_{\mathcal{J}_3}\\
&\underbrace{-\int_{0}^{t}\int_{\mathbb{T}^{3}}\nabla DH(U):A({U}^{\epsilon}|U)~dxds}_{\mathcal{J}_4}\underbrace{+\int_{0}^{t}\int_{\mathbb{T}^{3}}(\frac{{\rho}^{\epsilon}}{\rho}\rho_{f}-{\rho}^{\epsilon}_{f})(u_{f}-u)\cdot(u-{u}^{\epsilon})~dxds}_{\mathcal{J}_5}\\
&\underbrace{+\int_{0}^{t}\int_{\mathbb{T}^{3}}(\frac{\rho^{\epsilon}}{\rho}-1)\Delta u\cdot(u-{u}^{\epsilon})~dxds}_{\mathcal{J}_6}\underbrace{-\int_{0}^{t}\int_{\mathbb{T}^{3}}{\rho}^{\epsilon}_{f}(u_{f}-{u}^{\epsilon}_{f})\cdot\nabla K*(\rho_{f}-{\rho}^{\epsilon}_{f})~dxds}_{\mathcal{J}_7}\\
=:&\sum_{i=1}^{7}J_{i}.
\end{align*}
We estimate every terms $J_{i}\ (i=1,\cdots,7)$ as follows.
 \begin{itemize}
	\item (Estimate of $\mathcal{J}_1$): From assumption \eqref{O-assump-1}, it is clear that
	\begin{align*}
	\int_{\mathbb{T}^{3}}\mathcal{H}(U_0^\epsilon\vert U_0)dx\leq C\sqrt{\epsilon}.
	\end{align*}
	\item (Estimate of $\mathcal{J}_2$): Due to
	\begin{equation*}
	\frac{1}{2}\frac{d}{dt}\int_{\mathbb{T}^{3}}|\nabla K*(\rho_{f}^{\epsilon}-1)|^{2}~dx=\int_{\mathbb{T}^{3}}{\rho}^{\epsilon}_{f}{u}^{\epsilon}_{f}\cdot\nabla K*({\rho}^{\epsilon}_{f}-1)~dx,
	\end{equation*}	
	then by using Lemma \ref{O-modified entropy estimate}, we have
	\begin{align*}
	\mathcal{J}_2=&\int_{\mathbb{T}^{3}}H(U^{\epsilon})~dx-\int_{\mathbb{T}^{3}}H(U_{0}^{\epsilon})~dx+\int_{0}^{t}\int_{\mathbb{T}^{3}}|\nabla {u}^{\epsilon}|^{2}~dxds+\int_{0}^{t}\int_{\mathbb{T}^{3}}{\rho}^{\epsilon}_{f}|{u}^{\epsilon}_{f}-u|^{2}~dxds\\
	&+\frac{1}{2}\int_{\mathbb{T}^{3}}|\nabla K*(\rho_{f}^{\epsilon}-1)|^{2}~dx-\frac{1}{2}\int_{\mathbb{T}^{3}}|\nabla K*(\rho_{f_{0}}^{\epsilon}-1)|^{2}~dx\\
	 =&\frac{1}{2}\int_{\mathbb{T}^{3}}\rho^{\epsilon}_{f}|u^{\epsilon}_{f}|^{2}~dx+\frac{1}{2}\int_{\mathbb{T}^{3}}\rho^{\epsilon}|u^{\epsilon}|^{2}~dx+\int_{\mathbb{T}^{3}}\frac{1}{\gamma-1}(\rho^{\epsilon})^{\gamma}~dx+\int_{0}^{t}\int_{\mathbb{T}^{3}}|\nabla {u}^{\epsilon}|^{2}~dxds\\
	&+\int_{0}^{t}\int_{\mathbb{T}^{3}}{\rho}^{\epsilon}_{f}|{u}^{\epsilon}_{f}-u|^{2}~dxds+\frac{1}{2}\int_{\mathbb{T}^{3}}|\nabla K*(\rho_{f}^{\epsilon}-1)|^{2}~dx-\frac{1}{2}\int_{\mathbb{T}^{3}}|\nabla K*(\rho_{f_{0}}^{\epsilon}-1)|^{2}~dx\\
	 &-\int_{\mathbb{T}^{3}}H(U_{0}^{\epsilon})~dx+\frac{1}{2}\int_{\mathbb{T}^{3}\times\mathbb{R}^{3}}|\xi|^{2}f^{\epsilon}~dxd\xi-\frac{1}{2}\int_{\mathbb{T}^{3}\times\mathbb{R}^{3}}|\xi|^{2}f^{\epsilon}~dxd\xi\\
	=&\mathcal{F}(f^{\epsilon},\rho^{\epsilon},u^{\epsilon})(t)+\int_{0}^{t}\int_{\mathbb{T}^{3}}|\nabla {u}^{\epsilon}|^{2}~dxds+\int_{0}^{t}\int_{\mathbb{T}^{3}}{\rho}^{\epsilon}_{f}|{u}^{\epsilon}_{f}-u|^{2}~dxds\\
	&-\frac{1}{2}\int_{\mathbb{T}^{3}}|\nabla K*(\rho_{f_{0}}^{\epsilon}-1)|^{2}~dx-\int_{\mathbb{T}^{3}}H(U_{0}^{\epsilon})~dx
	 +\frac{1}{2}\int_{\mathbb{T}^{3}}\rho^{\epsilon}_{f}|u^{\epsilon}_{f}|^{2}~dx-\frac{1}{2}\int_{\mathbb{T}^{3}\times\mathbb{R}^{3}}|\xi|^{2}f^{\epsilon}~dxd\xi\\
	\leq& \mathcal{F}(f_{0}^{\epsilon},\rho_{0}^{\epsilon},u_{0}^{\epsilon})+C(T)\epsilon-\frac{1}{2}\int_{\mathbb{T}^{3}}|\nabla K*(\rho_{f_{0}}^{\epsilon}-1)|^{2}~dx-\int_{\mathbb{T}^{3}}H(U_{0}^{\epsilon})~dx\\
	\leq& C\sqrt{\epsilon},
	\end{align*}
	where we uesd the fact that $\rho^{\epsilon}_{f}|u^{\epsilon}_{f}|^{2}\leq\int_{\mathbb{R}^{3}}|\xi|^{2}f^{\epsilon}~d\xi$ and assumption \eqref{O-assump-2}.
	\item (Estimate of $\mathcal{J}_3$): It follows from \eqref{O-written-eq1.1} that
	\begin{equation}\label{O-MODI-system epsilon}
	\left\{
	\begin{array}{l}
	\partial_t{\rho^{\epsilon}_{f}}+\mathrm{div}_{x}{(\rho^{\epsilon}_{f}u^{\epsilon}_{f})}=0,\\
	\partial_{t}(\rho^{\epsilon}_{f}u^{\epsilon}_{f})+\mathrm{div}_{x}(\rho^{\epsilon}_{f}u^{\epsilon}_{f}\otimes u^{\epsilon}_{f})+\mathrm{div}_{x}\left(\int_{\mathbb{R}^{3}}(\xi\otimes\xi-u_{f}^{\epsilon}\otimes u_{f}^{\epsilon})f^{\epsilon}~d\xi\right)=\rho^{\epsilon}_{f}(u^{\epsilon}-u^{\epsilon}_{f})-\rho^{\epsilon}_{f}\nabla_{x}K*(\rho^{\epsilon}_{f}-1),\\
	\partial_{t}\rho^{\epsilon}+\mathrm{div}_{x}(\rho^{\epsilon} u^{\epsilon})=0,\\
	\partial_{t}(\rho^{\epsilon} u^{\epsilon})+\mathrm{div}_{x}(\rho^{\epsilon} u^{\epsilon}\otimes u^{\epsilon})+\nabla_{x}p^{\epsilon}-\Delta_{x} u^{\epsilon}=-\rho^{\epsilon}_{f}(u^{\epsilon}-u^{\epsilon}_{f}),
	\end{array}
	\right.
	\end{equation}
	in the sense of distributions. This implies
	\begin{align*}
	-\int_{0}^{t}&\int_{\mathbb{T}^{3}}DH(U)(\partial_{s}{U}^{\epsilon}+\mathrm{div}_{x}A({U}^{\epsilon})-F({U}^{\epsilon}))~dxds\\
	=&\int_{0}^{t}\int_{\mathbb{T}^{3}}u_{f}\cdot\mathrm{div}_{x}\left(\int_{\mathbb{R}^{3}}(\xi\otimes\xi-u_{f}^{\epsilon}\otimes u_{f}^{\epsilon})f^{\epsilon}~d\xi\right)dxds\\
	=&-\int_{0}^{t}\int_{\mathbb{T}^{3}}\nabla u_{f}:\left(\int_{\mathbb{R}^{3}}(\xi\otimes\xi-u_{f}^{\epsilon}\otimes u_{f}^{\epsilon})f^{\epsilon}~d\xi\right)dxds,
	\end{align*}
	then we have
	\begin{align*}
	\int_{0}^{t}&\int_{\mathbb{T}^{3}}\nabla u_{f}:\left(\int_{\mathbb{R}^{3}}(u_{f}^{\epsilon}\otimes u_{f}^{\epsilon}-\xi\otimes\xi)f^{\epsilon}~d\xi\right)dxds\\
	\leq& \|\nabla u_{f}\|_{L_{t,x}^{\infty}}\int_{0}^{t}\int_{\mathbb{T}^{3}}\left|\int_{\mathbb{R}^{3}}(u_{f}^{\epsilon}\otimes u_{f}^{\epsilon}-\xi\otimes\xi)f^{\epsilon}~d\xi\right|dxds\\
	\leq& C\int_{0}^{t}\int_{\mathbb{T}^{3}}\left|\int_{\mathbb{R}^{3}}(u_{f}^{\epsilon}\otimes(u_{f}^{\epsilon}-\xi)+(u_{f}^{\epsilon}-\xi) \otimes\xi)f^{\epsilon}~d\xi\right|dxds\\
	\leq& C\left(\int_{0}^{t}\int_{\mathbb{T}^{3}\times\mathbb{R}^{3}}|u_{f}^{\epsilon}|^{2}f^{\epsilon}~d\xi dxds\right)^{\frac{1}{2}}\left(\int_{0}^{t}\int_{\mathbb{T}^{3}\times\mathbb{R}^{3}}|u_{f}^{\epsilon}-\xi|^{2}f^{\epsilon}~d\xi dxds\right)^{\frac{1}{2}}\\
	&+\left(\int_{0}^{t}\int_{\mathbb{T}^{3}\times\mathbb{R}^{3}}|u_{f}^{\epsilon}-\xi|^{2}f^{\epsilon}~d\xi dxds\right)^{\frac{1}{2}}\left(\int_{0}^{t}\int_{\mathbb{T}^{3}\times\mathbb{R}^{3}}|\xi|^{2}f^{\epsilon}~d\xi dxds\right)^{\frac{1}{2}}\\
	=& C\left(\int_{0}^{t}\int_{\mathbb{T}^{3}\times\mathbb{R}^{3}}|u_{f}^{\epsilon}-\xi|^{2}f^{\epsilon}~d\xi dxds\right)^{\frac{1}{2}}\left[\left(\int_{0}^{t}\int_{\mathbb{T}^{3}\times\mathbb{R}^{3}}|u_{f}^{\epsilon}|^{2}f^{\epsilon}~d\xi dxds\right)^{\frac{1}{2}}+\left(\int_{0}^{t}\int_{\mathbb{T}^{3}\times\mathbb{R}^{3}}|\xi|^{2}f^{\epsilon}~d\xi dxds\right)^{\frac{1}{2}}\right]\\
	\leq&C\left(\int_{0}^{t}\int_{\mathbb{T}^{3}\times\mathbb{R}^{3}}|u_{f}^{\epsilon}-\xi|^{2}f^{\epsilon}~d\xi dxds\right)^{\frac{1}{2}}\left(\int_{0}^{t}\int_{\mathbb{T}^{3}\times\mathbb{R}^{3}}|\xi|^{2}f^{\epsilon}~d\xi dxds\right)^{\frac{1}{2}}\\
	\leq&C\sqrt{\epsilon},
	\end{align*}
	where $C=C(T,\|\nabla u_{f}\|_{L_{t,x}^{\infty}})>0$ is a constant independent of $\epsilon$.
	\item (Estimate of $\mathcal{J}_4$): By the definition of $A(U^{\epsilon}|U )$, one has
	\begin{center}
		$A(U^{\epsilon}|U)$=$\begin{pmatrix}
		0\\
		\rho_{f}^{\epsilon}(u_{f}^{\epsilon}-u_{f})\otimes(u_{f}^{\epsilon}-u_{f})\\
		0\\
		\rho^{\epsilon}(u^{\epsilon}-u)\otimes(u^{\epsilon}-u)+(\gamma-1)P(\rho^{\epsilon}|\rho)\mathbb{I}_{3\times3}
		\end{pmatrix}$,
	\end{center}
then
\begin{align*}
\mathcal{J}_4=&-\int_{0}^{t}\int_{\mathbb{T}^{3}}\rho_{f}^{\epsilon}(u_{f}^{\epsilon}-u_{f})\otimes(u_{f}^{\epsilon}-u_{f}):\nabla u_{f}~dxds\\
&-\int_{0}^{t}\int_{\mathbb{T}^{3}}\rho^{\epsilon}(u^{\epsilon}-u)\otimes(u^{\epsilon}-u):\nabla u+(\gamma-1)P(\rho^{\epsilon}|\rho)\mathbb{I}_{3\times3}:\nabla u~dxds\\
\leq &\|\nabla u_{f}\|_{L_{t,x}^{\infty}}\int_{0}^{t}\int_{\mathbb{T}^{3}}\rho_{f}^{\epsilon}|u_{f}^{\epsilon}-u_{f}|^{2}~dxds+\|\nabla u\|_{L_{t,x}^{\infty}}\int_{0}^{t}\int_{\mathbb{T}^{3}}\rho^{\epsilon}|u^{\epsilon}-u|^{2}+(\gamma-1)P(\rho^{\epsilon}|\rho)~dxds\\
\leq &C(\gamma,\|\nabla u_{f}\|_{L_{t,x}^{\infty}},\|\nabla u\|_{L_{t,x}^{\infty}})\int_{0}^{t}\int_{\mathbb{T}^{3}}\mathcal{H}(U^{\epsilon}|U)~dxds.
\end{align*}
\item (Estimate of $\mathcal{J}_5$):
\begin{align*}
\mathcal{J}_5=&\int_{0}^{t}\int_{\mathbb{T}^{3}}(\frac{{\rho}^{\epsilon}}{\rho}\rho_{f}-{\rho}^{\epsilon}_{f})(u_{f}-u)\cdot(u-{u}^{\epsilon})~dxds\\
=&\int_{0}^{t}\int_{\mathbb{T}^{3}}\rho_{f}(\frac{{\rho}^{\epsilon}-\rho}{\rho})(u_{f}-u)\cdot(u-{u}^{\epsilon})~dxds+\int_{0}^{t}\int_{\mathbb{T}^{3}}(\rho_{f}-\rho_{f}^{\epsilon })(u_{f}-u)\cdot(u-{u}^{\epsilon})~dxds\\
=&\mathcal{J}_{5}^{1}+\mathcal{J}_{5}^{2}.
\end{align*}
The proof idea is similar to \cite{Choi2021}, here for the reader's convenience,
we give the proof details.

Let $\rho_{*}:=\inf\limits_{(t,x)\in[0,T]\times\mathbb{T}^{3}} \rho(t,x)>0$, and first consider the case $\gamma\in(\frac{3}{2},2]$, we use the inequality \eqref{O-relative-pressure} to obtain
\begin{align*}
&\left|\int_{\mathbb{T}^{3}}\rho_{f}(\frac{{\rho}^{\epsilon}-\rho}{\rho})(u_{f}-u)\cdot(u-{u}^{\epsilon})~dx \right| \\
&\leq \|\frac{\rho_{f}}{\rho}\|_{L_{t,x}^{\infty}}\int_{\mathbb{T}^{3}}|{\rho}^{\epsilon}-\rho||u_{f}-u||u-{u}^{\epsilon}|~dx\\
&\leq \frac{\|\rho_{f}\|_{L_{t,x}^{\infty}}}{\rho_{*}}\int_{\mathbb{T}^{3}}|{\rho}^{\epsilon}-\rho||u_{f}-u||u-{u}^{\epsilon}|~dx\\
&\leq C(\|\rho_{f}\|_{L_{t,x}^{\infty}},\rho_{*}) \left(\int_{\mathbb{T}^{3}} \min\left\lbrace(\rho^{\epsilon})^{\gamma-2}, \rho^{\gamma-2}\right\rbrace(\rho^{\epsilon}-\rho)^{2} ~dx\right)^{\frac{1}{2}}\left(\int_{\mathbb{T}^{3}} \left((\rho^{\epsilon})^{2-\gamma}+\rho^{2-\gamma}\right) |u-{u}^{\epsilon}|^{2}|u_{f}-u|^{2} ~dx\right)^{\frac{1}{2}}\\
&\leq C(\|\rho_{f}\|_{L_{t,x}^{\infty}},\rho_{*},\gamma)\left(\int_{\mathbb{T}^{3}} \mathcal{H}(U^{\epsilon}|U) ~dx\right)^{\frac{1}{2}}\left(\int_{\mathbb{T}^{3}} \left((\rho^{\epsilon})^{2-\gamma}+\rho^{2-\gamma}\right) |u-{u}^{\epsilon}|^{2}|u_{f}-u|^{2} ~dx\right)^{\frac{1}{2}}.
\end{align*}
By using Poincar\'{e} type  inequality \eqref{O-lem poincare}, we have
\begin{align*}
	\int_{\mathbb{T}^{3}}& \left((\rho^{\epsilon})^{2-\gamma}+\rho^{2-\gamma}\right) |u-{u}^{\epsilon}|^{2}|u_{f}-u|^{2} ~dx\\
	&\leq \|\rho\|_{L_{t,x}^{\infty}}^{2-\gamma}\|u_{f}-u\|_{L_{x}^{\infty}}^{2}\|u-u^{\epsilon}\|_{L_{x}^{2}}^{2}+\int_{\mathbb{T}^{3}} (\rho^{\epsilon})^{2-\gamma} |u-{u}^{\epsilon}|^{2}|u_{f}-u|^{2} ~dx\\
	&\leq C \|\rho\|_{L_{t,x}^{\infty}}^{2-\gamma}\|u_{f}-u\|_{L_{x}^{\infty}}^{2}\left[ \bar{C}\int_{\mathbb{T}^{3}}\rho^{\epsilon}|u-u^{\epsilon}|^{2}~dx+\bar{C}\|\nabla(u-u^{\epsilon})\|^{2}_{L^{2}_{x}}\right]+\int_{\mathbb{T}^{3}} (\rho^{\epsilon})^{2-\gamma} |u-{u}^{\epsilon}|^{2}|u_{f}-u|^{2} ~dx\\
	&\leq C(\|\rho\|_{L_{t,x}^{\infty}},\gamma,\bar{C},\|u_{f}-u\|_{L_{t,x}^{\infty}}) \left( \int_{\mathbb{T}^{3}}\mathcal{H}(U^{\epsilon}|U)~dx+\|\nabla(u-u^{\epsilon})\|^{2}_{L^{2}_{x}}\right)+\int_{\mathbb{T}^{3}} (\rho^{\epsilon})^{2-\gamma} |u-{u}^{\epsilon}|^{2}|u_{f}-u|^{2} ~dx.
\end{align*}
For $\gamma=2$, by \eqref{O-lem poincare}, we obtain
\begin{align*}
	\int_{\mathbb{T}^{3}} (\rho^{\epsilon})^{2-\gamma} |u-{u}^{\epsilon}|^{2}|u_{f}-u|^{2} ~dx&\leq \|u_{f}-u\|_{L_{x}^{\infty}}^{2}\|u-u^{\epsilon}\|_{L_{x}^{2}}^{2}\\
	&\leq\|u_{f}-u\|_{L_{x}^{\infty}}^{2} \left[ \bar{C}\int_{\mathbb{T}^{3}}\rho^{\epsilon}|u-u^{\epsilon}|^{2}~dx+\bar{C}\|\nabla(u-u^{\epsilon})\|^{2}_{L^{2}_{x}}\right]\\
	 &\leq C(\bar{C},\|u_{f}-u\|_{L_{x}^{\infty}})\left( \int_{\mathbb{T}^{3}}\rho^{\epsilon}|u-u^{\epsilon}|^{2}~dx+\|\nabla(u-u^{\epsilon})\|^{2}_{L^{2}_{x}}\right)\\
	 &\leq C(\bar{C},\|u_{f}-u\|_{L_{x}^{\infty}})\left( \int_{\mathbb{T}^{3}}\mathcal{H}(U^{\epsilon}|U)~dx+\|\nabla(u-u^{\epsilon})\|^{2}_{L^{2}_{x}}\right).
\end{align*}
For $\gamma\in(\frac{3}{2},2)$, by using Young's inequality and \eqref{O-lem poincare}, one has
\begin{align*}
\int_{\mathbb{T}^{3}}& (\rho^{\epsilon})^{2-\gamma} |u-{u}^{\epsilon}|^{2}|u_{f}-u|^{2} ~dx\\
&=\int_{\mathbb{T}^{3}} (\rho^{\epsilon})^{2-\gamma} |u-{u}^{\epsilon}|^{(4-2\gamma)+(2\gamma-2)}|u_{f}-u|^{2} ~dx\\
&\leq(2-\gamma)\int_{\mathbb{T}^{3}} \rho^{\epsilon} |u-{u}^{\epsilon}|^{2}~dx+(\gamma-1)\int_{\mathbb{T}^{3}}|u-u^{\epsilon}|^{2}|u_{f}-u|^{\frac{2}{\gamma-1}}dx\\
&\leq(2-\gamma)\int_{\mathbb{T}^{3}} \rho^{\epsilon} |u-{u}^{\epsilon}|^{2}~dx+(\gamma-1)\|u_{f}-u\|^{\frac{2}{\gamma-1}}_{L_{x}^{\infty}}\int_{\mathbb{T}^{3}}|u-u^{\epsilon}|^{2}dx\\
&\leq(2-\gamma)\int_{\mathbb{T}^{3}} \rho^{\epsilon} |u-{u}^{\epsilon}|^{2}~dx+(\gamma-1)\|u_{f}-u\|^{\frac{2}{\gamma-1}}_{L_{x}^{\infty}}\left[ \bar{C}\int_{\mathbb{T}^{3}}\rho^{\epsilon}|u-u^{\epsilon}|^{2}~dx+\bar{C}\|\nabla(u-u^{\epsilon})\|^{2}_{L^{2}_{x}}\right]\\
&\leq C(\gamma,\bar{C},\|u_{f}-u\|_{L_{t,x}^{\infty}})\left(\int_{\mathbb{T}^{3}} \mathcal{H}(U^{\epsilon}|U) ~dx+\|\nabla(u-u^{\epsilon})\|_{L_{x}^{2}}^{2}\right),
\end{align*}
thus, for $\gamma\in(\frac{3}{2},2]$, we have
\begin{align*}
\mathcal{J}_{5}^{1}&\leq C\int_{0}^{t}\left(\int_{\mathbb{T}^{3}} \mathcal{H}(U^{\epsilon}|U) ~dx\right)^{\frac{1}{2}}\left(\int_{\mathbb{T}^{3}} \mathcal{H}(U^{\epsilon}|U) ~dx+\|\nabla(u-u^{\epsilon})\|_{L_{x}^{2}}^{2}\right)^{\frac{1}{2}}\\
&\leq C\int_{0}^{t}\int_{\mathbb{T}^{3}} \mathcal{H}(U^{\epsilon}|U) ~dxds+\frac{1}{8}\int_{0}^{t}\int_{\mathbb{T}^{3}}|\nabla(u-u^{\epsilon})|^{2}~dxds
\end{align*}
where $C=C(\bar{C},\|\rho_{f}\|_{L_{t,x}^{\infty}},\rho_{*},\|\rho\|_{L_{t,x}^{\infty}},\gamma,\|u_{f}-u\|_{L_{t,x}^{\infty}})>0$.

For $\gamma>2$, we use Lemma \ref{O-Lem-relative} with $x=\rho^{\epsilon}$, $y=\rho$, $y_{\min}=\rho_{*}$, $y_{\max}=\|\rho\|_{L_{t,x}^{\infty}}$, one has
\begin{align*}
		&\left|\int_{\mathbb{T}^{3}}\rho_{f}(\frac{{\rho}^{\epsilon}-\rho}{\rho})(u_{f}-u)\cdot(u-{u}^{\epsilon})~dx \right| \\
		&\leq \frac{\|\rho_{f}(u_{f}-u)\|_{L_{t,x}^{\infty}}}{\rho_{*}}\int_{\mathbb{T}^{3}}|{\rho}^{\epsilon}-\rho||u-{u}^{\epsilon}|~dx\\
		&\leq C\left(\int_{\mathbb{T}^{3}\cap\left\lbrace \frac{\rho}{2}\leq \rho^{\epsilon}\leq 2\rho\right\rbrace }+\int_{\mathbb{T}^{3}\cap\left\lbrace \frac{\rho}{2}\leq \rho^{\epsilon}\leq 2\rho\right\rbrace^{c} }\right)|{\rho}^{\epsilon}-\rho||u-{u}^{\epsilon}|~dx\\
		&=:\mathcal{J}_{5}^{11}+\mathcal{J}_{5}^{12},
\end{align*}
where $C=C(\rho_{*},\|\rho_{f}\|_{L_{t,x}^{\infty}},\|u_{f}-u\|_{L_{t,x}^{\infty}})>0$.
\begin{align*}
\mathcal{J}_{5}^{11}&\leq \left(\int_{\mathbb{T}^{3}\cap\left\lbrace \frac{\rho}{2}\leq \rho^{\epsilon}\leq 2\rho\right\rbrace } \frac{|\rho^{\epsilon}-\rho|^{2}}{\rho^{\epsilon}}~dx\right)^{\frac{1}{2}}\left(\int_{\mathbb{T}^{3}\cap\left\lbrace \frac{\rho}{2}\leq \rho^{\epsilon}\leq 2\rho\right\rbrace } \rho^{\epsilon}|u^{\epsilon}-u|^{2}~dx\right)^{\frac{1}{2}}\\
&\leq \left(\int_{\mathbb{T}^{3}\cap\left\lbrace \frac{\rho}{2}\leq \rho^{\epsilon}\leq 2\rho\right\rbrace } \frac{|\rho^{\epsilon}-\rho|^{2}}{\frac{\rho}{2}}~dx\right)^{\frac{1}{2}}\left(\int_{\mathbb{T}^{3}}  \mathcal{H}(U^{\epsilon}|U)~dx\right)^{\frac{1}{2}}\\
&\leq C(\rho_{*},\|\rho\|_{L_{t,x}^{\infty}},\gamma)\int_{\mathbb{T}^{3}}  \mathcal{H}(U^{\epsilon}|U)~dx.
\end{align*}
For $\mathcal{J}_{5}^{12}$, $\mathbb{T}^{3}\cap\left\lbrace \frac{\rho}{2}\leq \rho^{\epsilon}\leq 2\rho\right\rbrace^{c}=\mathbb{T}^{3}\cap\left\lbrace \rho^{\epsilon}> 2\rho\right\rbrace\cup\left\lbrace \rho^{\epsilon}< \frac{\rho}{2}\right\rbrace$. On the region $\left\lbrace \rho^{\epsilon}> 2\rho\right\rbrace$,  we obtain
\begin{align*}
&\int_{\mathbb{T}^{3}\cap\left\lbrace \rho^{\epsilon}> 2\rho\right\rbrace}|\rho^{\epsilon}-\rho||u^{\epsilon}-u|~dx\\
&\leq \left(\int_{\mathbb{T}^{3}} \min\left\lbrace(\rho^{\epsilon})^{\gamma-2},\rho^{\gamma-2} \right\rbrace(\rho-\rho^{\epsilon})^{2}~dx \right)^{\frac{1}{2}}\left(\int_{\mathbb{T}^{3}\cap\left\lbrace \rho^{\epsilon}> 2\rho\right\rbrace}( (\rho^{\epsilon})^{2-\gamma}+\rho^{2-\gamma}) |u-u^{\epsilon}| ^{2}~dx \right)^{\frac{1}{2}}\\
&\leq \left( \int_{\mathbb{T}^{3}} \mathcal{H}(U^{\epsilon}|U)~dx\right) ^{\frac{1}{2}}\left(\int_{\mathbb{T}^{3}\cap\left\lbrace \rho^{\epsilon}> 2\rho\right\rbrace}((2\rho)^{2-\gamma}+\rho^{2-\gamma}) |u-u^{\epsilon}| ^{2}~dx \right)^{\frac{1}{2}}\\
&\leq C\left(\int_{\mathbb{T}^{3}} \mathcal{H}(U^{\epsilon}|U)~dx\right) ^{\frac{1}{2}}\|u-u^{\epsilon}\|_{L_{x}^{2}}\\
&\leq C \left( \int_{\mathbb{T}^{3}} \mathcal{H}(U^{\epsilon}|U)~dx\right) ^{\frac{1}{2}}\left( \bar{C}\int_{\mathbb{T}^{3}}\rho^{\epsilon}|u-u^{\epsilon}|^{2}~dx+\bar{C}\|\nabla(u-u^{\epsilon})\|^{2}_{L^{2}_{x}}\right)^{\frac{1}{2}}\\
&\leq C  \int_{\mathbb{T}^{3}} \mathcal{H}(U^{\epsilon}|U)~dx +\frac{1}{16}\|\nabla(u-u^{\epsilon})\|^{2}_{L_{x}^{2}},
\end{align*}
where $C=C(\gamma,\mathbb{T}^{3},\rho_{*},\bar{C},\|\rho\|_{L_{t,x}^{\infty}})>0$.

On the region $\mathbb{T}^{3}\cap\left\lbrace  \rho^{\epsilon}<\frac{\rho}{2}\right\rbrace$, we note that $|\rho^{\epsilon}-\rho|=\rho-\rho^{\epsilon}>\frac{\rho}{2}$. Thus, by using Cauchy-Schwarz inequality, Sobolev inequality and Lemma \ref{O-Lem-relative}, one has
\begin{align*}
	&\int_{\mathbb{T}^{3}\cap\left\lbrace \rho^{\epsilon} <\frac{\rho}{2}\right\rbrace}|\rho^{\epsilon}-\rho||u^{\epsilon}-u|~dx\\
	&\leq C\left(\int_{\mathbb{T}^{3}\cap\left\lbrace \rho^{\epsilon} <\frac{\rho}{2}\right\rbrace}|\rho^{\epsilon}-\rho|^{2}~dx\right)^{\frac{1}{2}}\|u^{\epsilon}-u\|_{L^{2}_{x}}\\
	&\leq C\left(\int_{\mathbb{T}^{3}\cap\left\lbrace \rho^{\epsilon} <\frac{\rho}{2}\right\rbrace}\frac{|\rho^{\epsilon}-\rho|^{\gamma}}{|\rho^{\epsilon}-\rho|^{\gamma-2}}~dx\right)^{\frac{1}{2}}\|u^{\epsilon}-u\|_{L^{2}_{x}}\\
	&\leq C\left(\int_{\mathbb{T}^{3}\cap\left\lbrace \rho^{\epsilon} <\frac{\rho}{2}\right\rbrace}\frac{|\rho^{\epsilon}-\rho|^{\gamma}}{(\frac{\rho}{2})^{\gamma-2}}~dx\right)^{\frac{1}{2}}\|u^{\epsilon}-u\|_{L^{2}_{x}}\\
	&\leq C\left(\int_{\mathbb{T}^{3}\cap\left\lbrace \rho^{\epsilon} <\frac{\rho}{2}\right\rbrace}\frac{\|\rho\|_{L^{\infty}}}{(\frac{\rho^{*}}{2})^{\gamma-2}}\left|\frac{\rho^{\epsilon}}{\rho}+1\right|^{\gamma}~dx\right)^{\frac{1}{2}}\|u^{\epsilon}-u\|_{L^{2}_{x}}\\
	&\leq C\left(\int_{\mathbb{T}^{3}\cap\left\lbrace \rho^{\epsilon} <\frac{\rho}{2}\right\rbrace}(1+(\rho^{\epsilon})^{\gamma})~dx\right)^{\frac{1}{2}}\left( \bar{C}\int_{\mathbb{T}^{3}}\rho^{\epsilon}|u-u^{\epsilon}|^{2}~dx+\bar{C}\|\nabla(u-u^{\epsilon})\|^{2}_{L^{2}_{x}}\right)^{\frac{1}{2}}\\
	&\leq C\left(\int_{\mathbb{T}^{3}\cap\left\lbrace \rho^{\epsilon} <\frac{\rho}{2}\right\rbrace}\mathcal{H}(U^{\epsilon}|U)~dx\right)^{\frac{1}{2}}\left( \bar{C}\int_{\mathbb{T}^{3}}\mathcal{H}(U^{\epsilon}|U)~dx+\bar{C}\|\nabla(u-u^{\epsilon})\|^{2}_{L^{2}_{x}}\right)^{\frac{1}{2}}\\
	&\leq C\int_{\mathbb{T}^{3}}\mathcal{H}(U^{\epsilon}|U)~dx+\frac{1}{16}\|\nabla(u^{\epsilon}-u)\|_{L^{2}_{x}}^{2}.
\end{align*}
Therefore, for $\gamma>\frac{3}{2}$, one has
\begin{align*}
\mathcal{J}_{5}^{1}
&\leq C\int_{0}^{t}\int_{\mathbb{T}^{3}} \mathcal{H}(U^{\epsilon}|U) ~dxds+\frac{1}{4}\int_{0}^{t}\int_{\mathbb{T}^{3}}|\nabla(u-u^{\epsilon})|^{2}~dxds.
\end{align*}
where $C=C(\gamma,\rho_{*},\bar{C},\|\rho\|_{L_{t,x}^{\infty}},\|\rho_{f}\|_{L_{t,x}^{\infty}},\|u_{f}-u\|_{L_{t,x}^{\infty}},\mathbb{T}^{3})>0$.

For $\mathcal{J}_{5}^{2}$,
\begin{align}\label{O-J52}
\mathcal{J}_{5}^{2}=\int_{0}^{t}\int_{\mathbb{T}^{3}}(\rho_{f}^{\epsilon }-\rho_{f})(u_{f}-u)\cdot({u}^{\epsilon}-u)~dxds,
\end{align}
since $-\Delta K*(\rho_{f}^{\epsilon}-\rho_{f})=\rho_{f}^{\epsilon}-\rho_{f}$ in the weak sense, we have
\begin{align*}
\mathcal{J}_{5}^{2}=&-\int_{0}^{t}\int_{\mathbb{T}^{3}}\Delta K*(\rho_{f}^{\epsilon}-\rho_{f})(u_{f}-u)\cdot({u}^{\epsilon}-u)~dxds\\
=&\int_{0}^{t}\int_{\mathbb{T}^{3}}\nabla K*(\rho_{f}^{\epsilon}-\rho_{f})\cdot\nabla((u_{f}-u)\cdot({u}^{\epsilon}-u))~dxds\\
=&\int_{0}^{t}\int_{\mathbb{T}^{3}}\nabla K*(\rho_{f}^{\epsilon}-\rho_{f})\cdot(\nabla(u_{f}-u)\cdot({u}^{\epsilon}-u)+(u_{f}-u)\cdot\nabla({u}^{\epsilon}-u))~dxds\\
\leq& C\int_{0}^{t}\int_{\mathbb{T}^{3}}|\nabla K*(\rho_{f}^{\epsilon}-\rho_{f})|(|{u}^{\epsilon}-u|+|\nabla({u}^{\epsilon}-u)|)~dxds\\
\leq& C\int_{0}^{t}\|\nabla K*(\rho_{f}^{\epsilon}-\rho_{f})\|_{L_{x}^{2}}\|\nabla({u}^{\epsilon}-u)\|_{L_{x}^{2}}~ds+C\int_{0}^{t}\|\nabla K*(\rho_{f}^{\epsilon}-\rho_{f})\|_{L_{x}^{2}}\|{u}^{\epsilon}-u\|_{L_{x}^{2}}~ds\\
\leq& C\int_{0}^{t}\|\nabla K*(\rho_{f}^{\epsilon}-\rho_{f})\|^{2}_{L_{x}^{2}}~ds+\frac{1}{16}\int_{0}^{t}\|\nabla({u}^{\epsilon}-u)\|^{2}_{L_{x}^{2}}~ds\\
&+\int_{0}^{t}\|\nabla K*(\rho_{f}^{\epsilon}-\rho_{f})\|_{L_{x}^{2}}\left( \bar{C}\int_{\mathbb{T}^{3}}\mathcal{H}(U^{\epsilon}|U)~dx+\bar{C}\|\nabla(u-u^{\epsilon})\|^{2}_{L^{2}_{x}}\right)^{\frac{1}{2}}~ds\\
\leq& C\int_{0}^{t}\|\nabla K*(\rho_{f}^{\epsilon}-\rho_{f})\|^{2}_{L_{x}^{2}}~ds+\frac{1}{8}\int_{0}^{t}\|\nabla({u}^{\epsilon}-u)\|^{2}_{L_{x}^{2}}~ds+C\int_{0}^{t}\int_{\mathbb{T}^{3}}\mathcal{H}(U^{\epsilon}|U)~dxds,
\end{align*}
where $C=C(\bar{C},\|\nabla(u_{f}-u)\|_{L_{t,x}^{\infty}},\|u_{f}-u\|_{L_{t,x}^{\infty}})>0$.
 Thus,
 \begin{align*}
 	\mathcal{J}_5\leq C\int_{0}^{t}\int_{\mathbb{T}^{3}} \mathcal{H}(U^{\epsilon}|U) ~dxds+\frac{3}{8}\int_{0}^{t} \int_{\mathbb{T}^{3}} |\nabla(u-u^{\epsilon})|^{2}~dxds+C\int_{0}^{t}\|\nabla K*(\rho_{f}^{\epsilon}-\rho_{f})\|^{2}_{L_{x}^{2}}~ds,
 \end{align*}
 where $C=C(\|\rho_{f}\|_{L_{t,x}^{\infty}},\rho_{*},\gamma,\bar{C},\|\rho\|_{L_{t,x}^{\infty}},\|u_{f}-u\|_{L_{t,x}^{\infty}},\|\nabla(u_{f}-u)\|_{L_{t,x}^{\infty}})>0$.

\item (Estimate of $\mathcal{J}_6$): Similar to the estimate $\mathcal{J}^{1}_5$. For $\gamma\in(\frac{3}{2},2]$,
\begin{align*}
\mathcal{J}_6&=\int_{0}^{t}\int_{\mathbb{T}^{3}}(\frac{\rho^{\epsilon}}{\rho}-1)\Delta u\cdot(u-u^{\epsilon})~dxds\\
&\leq \frac{1}{\rho_{*}}\int_{0}^{t}\int_{\mathbb{T}^{3}}|\rho^{\epsilon}-\rho||u-u^{\epsilon}||\Delta u|~dxds\\
&\leq C\int_{0}^{t}\left( \int_{\mathbb{T}^{3}} \min\left\lbrace(\rho^{\epsilon})^{\gamma-2},\rho^{\gamma-2} \right\rbrace (\rho-\rho^{\epsilon})^{2}~dx\right)^{\frac{1}{2}}\left( \int_{\mathbb{T}^{3}}(\rho^{2-\gamma}+(\rho^{\epsilon})^{2-\gamma})  |u-u^{\epsilon}|^{2}|\Delta u|^{2}~dx\right)^{\frac{1}{2}}ds\\
&\leq C\int_{0}^{t}\int_{\mathbb{T}^{3}} \mathcal{H}(U^{\epsilon}|U)~dxds+\frac{1}{16}\int_{0}^{t} \int_{\mathbb{T}^{3}} |\nabla(u-u^{\epsilon})|^{2}~dxds.
\end{align*}
For $\gamma>2$,
\begin{align*}
\mathcal{J}_6&=\left|  \int_{0}^{t}\int_{\mathbb{T}^{3}}(\frac{\rho^{\epsilon}}{\rho}-1)\Delta u\cdot(u-u^{\epsilon})~dxds\right|\\
&\leq \int_{0}^{t}\frac{\|\Delta u\|_{L_{x}^{\infty}}}{\rho_{*}}\int_{\mathbb{T}^{3}}|\rho^{\epsilon}-\rho||u-u^{\epsilon}||\Delta u|~dxds\\
&\leq C\int_{0}^{t}\int_{\mathbb{T}^{3}} \mathcal{H}(U^{\epsilon}|U)~dxds+\frac{1}{16}\int_{0}^{t} \int_{\mathbb{T}^{3}} |\nabla(u-u^{\epsilon})|^{2}~dxds.
\end{align*}
 Thus, for $\gamma>\frac{3}{2}$, we obtain
\begin{align*}
\mathcal{J}_6\leq C\int_{0}^{t}\int_{\mathbb{T}^{3}} \mathcal{H}(U^{\epsilon}|U)~dxds+\frac{1}{8}\int_{0}^{t} \int_{\mathbb{T}^{3}} |\nabla(u-u^{\epsilon})|^{2}~dxds,
\end{align*}
where $C=C(\rho_{*},\bar{C},\|\Delta u\|_{L_{t,x}^{\infty}})>0$.
\item (Estimate of $\mathcal{J}_7$):
Since
\begin{align*}
\frac{1}{2}\frac{d}{dt}\int_{\mathbb{T}^{3}}|\nabla K*(\rho_{f}^{\epsilon}-\rho_{f})|^{2}~dx&=\int_{\mathbb{T}^{3}}\nabla K*(\rho_{f}^{\epsilon}-\rho_{f})\cdot(\nabla K*(\partial_{t}\rho_{f}^{\epsilon}-\partial_{t}\rho_{f}))~dx\\
&=-\int_{\mathbb{T}^{3}}\Delta K*(\rho_{f}^{\epsilon}-\rho_{f})( K*(\partial_{t}\rho_{f}^{\epsilon}-\partial_{t}\rho_{f}))~dx\\
&=\int_{\mathbb{T}^{3}}(\rho_{f}^{\epsilon}-\rho_{f})( K*(\partial_{t}\rho_{f}^{\epsilon}-\partial_{t}\rho_{f}))~dx.
\end{align*}
we use the symmetry of $K$ to get
\begin{align*}
\int_{\mathbb{T}^{3}}&(\rho_{f}^{\epsilon}-\rho_{f})( K*(\partial_{t}\rho_{f}^{\epsilon}-\partial_{t}\rho_{f}))~dx\\
&=-\int_{\mathbb{T}^{3}\times\mathbb{T}^{3}}(\rho_{f}^{\epsilon}-\rho_{f})(x) K(x-y)(\mathrm{div}_{y}(\rho_{f}^{\epsilon}u_{f}^{\epsilon})(y)-\mathrm{div}_{y}(\rho_{f}u_{f})(y))~dxdy\\
&=\int_{\mathbb{T}^{3}\times\mathbb{T}^{3}}(\rho_{f}^{\epsilon}-\rho_{f})(x) \nabla_{y}K(x-y)((\rho_{f}^{\epsilon}u_{f}^{\epsilon})(y)-(\rho_{f}u_{f})(y))~dxdy\\
&=-\int_{\mathbb{T}^{3}\times\mathbb{T}^{3}}(\rho_{f}^{\epsilon}-\rho_{f})(x) \nabla_{x}K(x-y)((\rho_{f}^{\epsilon}u_{f}^{\epsilon})(y)-(\rho_{f}u_{f})(y))~dxdy\\
&=\int_{\mathbb{T}^{3}\times\mathbb{T}^{3}}(\rho_{f}^{\epsilon}-\rho_{f})(y) \nabla K(x-y)((\rho_{f}^{\epsilon}u_{f}^{\epsilon})(x)-(\rho_{f}u_{f})(x))~dxdy\\
&=\int_{\mathbb{T}^{3}}\nabla K*(\rho_{f}^{\epsilon}-\rho_{f})\cdot (\rho_{f}^{\epsilon}u_{f}^{\epsilon}-\rho_{f}u_{f})~dx.
\end{align*}
then
\begin{align*}
\mathcal{J}_7&+\int_{0}^{t}\int_{\mathbb{T}^{3}}{\rho}^{\epsilon}_{f}{u}^{\epsilon}_{f}\cdot\nabla K*(\rho^{\epsilon}_{f}-{\rho}_{f})~dxds-\int_{0}^{t}\int_{\mathbb{T}^{3}}{\rho}_{f}{u}_{f}\cdot\nabla K*(\rho^{\epsilon}_{f}-{\rho}_{f})~dxds\\
&=\int_{0}^{t}\int_{\mathbb{T}^{3}}{\rho}^{\epsilon}_{f}{u}_{f}\cdot\nabla K*(\rho^{\epsilon}_{f}-{\rho}_{f})~dxds-\int_{0}^{t}\int_{\mathbb{T}^{3}}{\rho}_{f}{u}_{f}\cdot\nabla K*(\rho^{\epsilon}_{f}-{\rho}_{f})~dxds\\
&=\int_{0}^{t}\int_{\mathbb{T}^{3}}({\rho}^{\epsilon}_{f}-\rho_{f}){u}_{f}\cdot\nabla K*(\rho^{\epsilon}_{f}-{\rho}_{f})~dxds\\
&=-\int_{0}^{t}\int_{\mathbb{T}^{3}}\Delta K*({\rho}^{\epsilon}_{f}-\rho_{f}){u}_{f}\cdot\nabla K*(\rho^{\epsilon}_{f}-{\rho}_{f})~dxds\\
&=\int_{0}^{t}\int_{\mathbb{T}^{3}}\nabla K*({\rho}^{\epsilon}_{f}-\rho_{f})\otimes\nabla K*(\rho^{\epsilon}_{f}-{\rho}_{f}):\nabla{u}_{f}~dxds-\frac{1}{2}\int_{0}^{t}\int_{\mathbb{T}^{3}}|\nabla K*({\rho}^{\epsilon}_{f}-\rho_{f})|^{2}\mathrm{div}_{x}u_{f}~dxds\\
&\leq C\int_{0}^{t}\int_{\mathbb{T}^{3}}|\nabla K*({\rho}^{\epsilon}_{f}-\rho_{f})|^{2}~dxds
\end{align*}
where $C=C(\|\nabla{u}_{f}\|_{L^{\infty}_{t,x}})>0$.
\end{itemize}
In conclusion, we have
\begin{align*}
\int_{\mathbb{T}^{3}}&\mathcal{H}(U^{\epsilon}|U)~dx+\frac{1}{2}\int_{\mathbb{T}^{3}}|\nabla K*({\rho}^{\epsilon}_{f}-\rho_{f})|^{2}~dx+\int_{0}^{t}\int_{\mathbb{T}^{3}}{\rho}^{\epsilon}_{f}|({u}^{\epsilon}_{f}-{u}^{\epsilon})-({u}_{f}-{u})|^{2}~dxds\\
&+\frac{1}{2}\int_{0}^{t}\int_{\mathbb{T}^{3}}|\nabla(u-{u}^{\epsilon})|^{2}~dxds\leq C\int_{0}^{t}\int_{\mathbb{T}^{3}} \mathcal{H}(U^{\epsilon}|U)~dxds+C\int_{0}^{t}\int_{\mathbb{T}^{3}}|\nabla K*({\rho}^{\epsilon}_{f}-\rho_{f})|^{2}~dxds+C\sqrt{\epsilon}.
\end{align*}
By using Gr\"{o}nwall's inequality, we obtain
\begin{align*}
\int_{\mathbb{T}^{3}}&\mathcal{H}(U^{\epsilon}|U)~dx+\frac{1}{2}\int_{\mathbb{T}^{3}}|\nabla K*({\rho}^{\epsilon}_{f}-\rho_{f})|^{2}~dx+\int_{0}^{t}\int_{\mathbb{T}^{3}}{\rho}^{\epsilon}_{f}|({u}^{\epsilon}_{f}-{u}^{\epsilon})-({u}_{f}-{u})|^{2}~dxds\\
&+\frac{1}{2}\int_{0}^{t}\int_{\mathbb{T}^{3}}|\nabla(u-{u}^{\epsilon})|^{2}~dxds\leq C\sqrt{\epsilon},
\end{align*}
where $C=C(T,\|\rho_{f}\|_{L_{t,x}^{\infty}},\rho_{*},\gamma,\bar{C},\|\rho\|_{L_{t,x}^{\infty}},\|u_{f}-u\|_{L_{t,x}^{\infty}},\|\nabla(u_{f}-u)\|_{L_{t,x}^{\infty}},\|\Delta u\|_{L_{t,x}^{\infty}},\|\nabla u_{f}\|_{L_{t,x}^{\infty}})>0$.

\subsection{Passing to the limit}\label{O-Section2-4}
We now conisder the convergence of each physical quantity. For the convergence of
$\rho^{\epsilon}$, $\rho^{\epsilon}u^{\epsilon}$, $\rho^{\epsilon}u^{\epsilon}\otimes u^{\epsilon}$, which can be obtained by the similar argument in Ref. \cite{Choi2021}.	Now we consider the convergence of $\rho_{f}^{\epsilon}$, $\rho_{f}^{\epsilon}u_{f}^{\epsilon}$, $\rho_{f}^{\epsilon}u_{f}^{\epsilon}\otimes u_{f}^{\epsilon}$.

At first, we give the convergence of $\rho_{f}^{\epsilon}$. For any test functions $\phi\in L^{2}((0,T;H^{1}{(\mathbb{T}^{3})})$, since $-\Delta K*(\rho_{f}^{\epsilon}-\rho_{f})=\rho_{f}^{\epsilon}-\rho_{f}$ in the weak sense, we have
 \begin{align*}
 \int_{0}^{T}&\int_{\mathbb{T}^{3}}(\rho_{f}^{\epsilon}- \rho_{f} )\phi~dxdt\nonumber\\
 =&-\int_{0}^{T}\int_{\mathbb{T}^{3}}\Delta K*(\rho_{f}^{\epsilon}-\rho_{f}) \phi~dxdt\\
 =&\int_{0}^{T}\int_{\mathbb{T}^{3}}\nabla K*(\rho_{f}^{\epsilon}-\rho_{f})\cdot \nabla\phi~dxdt\\
 \leq &\|\nabla K*(\rho_{f}^{\epsilon}-\rho_{f})\|_{L_{t,x}^{2}}\|\nabla \phi\|_{L_{t,x}^{2}}\\
 \leq & C\sqrt{\epsilon},
 \end{align*}
 thus, as $\epsilon\rightarrow 0$, one has
 \begin{equation}
 \rho_{f}^{\epsilon}\longrightarrow \rho_{f} \ \ \ \mathrm{in}  \ \ \  L^{2}(0,T; H^{-1}(\mathbb{T}^{3})).
 \end{equation}

\noindent Next, we give the convergence of $\rho_{f}^{\epsilon}u_{f}^{\epsilon}$. For any test functions $\phi\in C_{c}^{\infty}(\mathbb{T}^{3}\times\mathbb{R}^{3})$, one obtains
\begin{align*}
\int_{0}^{T}&\int_{\mathbb{T}^{3}}(\rho_{f}^{\epsilon}u_{f}^{\epsilon}-\rho_{f} u_{f})\varphi~dxdt\nonumber\\
=&\int_{0}^{T}\int_{\mathbb{T}^{3}}(\rho_{f}^{\epsilon}-\rho_{f})u_{f} \varphi~dxdt+\int_{0}^{T}\int_{\mathbb{T}^{3}}\rho_{f}^{\epsilon}(u_{f}^{\epsilon}-u_{f})\varphi~dxdt\nonumber\\
\leq&\left(\int_{0}^{T}\int_{\mathbb{T}^{3}}\rho^{\epsilon}_{f}~dxds\right)^{\frac{1}{2}}\left(\int_{0}^{T}\int_{\mathbb{T}^{3}}\rho^{\epsilon}_{f}|u_{f}^{\epsilon}-u_{f}|^{2}dxdt\right)^{\frac{1}{2}} \|\varphi\|_{L_{t,x}^{\infty}}+\int_{0}^{T}\int_{\mathbb{T}^{3}}(\rho_{f}^{\epsilon}-\rho_{f})u_{f} \varphi~dxdt\nonumber\\
&\longrightarrow  0 (\epsilon\longrightarrow 0),
\end{align*}
thus
\begin{equation}
	\rho_{f}^{\epsilon}u_{f}^{\epsilon}\longrightarrow \rho_{f} u_{f} \ \ \ \mathrm{in} \ \ \  \mathcal{D}'((0,T)\times\mathbb{T}^{3}).
\end{equation}

Finally, we consider the convergence of
$f^{\epsilon}\rightarrow \rho_{f}\otimes\delta_{\xi=u_{f}}$ as $\epsilon\rightarrow 0$.
For any test functions $\phi\in C_{c}^{\infty}(\mathbb{T}^{3}\times\mathbb{R}^{3})$, one has
\begin{align*}
\int_{0}^{T}&\int_{\mathbb{T}^{3}\times\mathbb{R}^{3}}({f}^{\epsilon}- \rho_{f}\otimes\delta_{\xi=u_{f}})\varphi~d\xi dxdt\nonumber\\
=&\int_{0}^{T}\int_{\mathbb{T}^{3}\times\mathbb{R}^{3}} \varphi(x,\xi)f^{\epsilon}~dxdt-\int_{0}^{T}\int_{\mathbb{T}^{3}\times\mathbb{R}^{3}}\varphi(x,\xi)\rho_{f}\otimes\delta_{\xi=u_{f}}~d\xi dxdt\nonumber\\
=&\int_{0}^{T}\int_{\mathbb{T}^{3}\times\mathbb{R}^{3}} \varphi(x,\xi)f^{\epsilon}~dxdt-\int_{0}^{T}\int_{\mathbb{T}^{3}}\varphi(x,u_{f})\rho_{f}~dxdt\nonumber\\
=&\int_{0}^{T}\int_{\mathbb{T}^{3}\times\mathbb{R}^{3}} (\varphi(x,\xi)-\varphi(x,u_{f}))f^{\epsilon}~dxdt+\int_{0}^{T}\int_{\mathbb{T}^{3}}\varphi(x,u_{f})(\rho_{f}^{\epsilon}-\rho_{f})~dxdt\nonumber\\
=:&A_{1}+A_{2}.
\end{align*}
For $A_{1}$,
\begin{align*}
A_{1}\leq& \|\nabla_{\xi}\varphi\|_{L^{\infty}}\int_{0}^{T}\int_{\mathbb{T}^{3}\times\mathbb{R}^{3}} |\xi-u_{f}|f^{\epsilon}~d\xi dxdt\\
\leq& C\left(\int_{0}^{T}\int_{\mathbb{T}^{3}\times\mathbb{R}^{3}} |\xi-u_{f}|^{2}f^{\epsilon}~d\xi dxdt\right)^{\frac{1}{2}}\left(\int_{0}^{T}\int_{\mathbb{T}^{3}\times\mathbb{R}^{3}} f^{\epsilon}~d\xi dxdt\right)^{\frac{1}{2}}\\
\leq&C\sqrt{\epsilon}\longrightarrow 0 \ (\epsilon\longrightarrow 0),
\end{align*}
where we used the fact that
\begin{align*}
 \int_{0}^{T}&\int_{\mathbb{T}^{3}\times\mathbb{R}^{3}} |\xi-u_{f}|^{2}f^{\epsilon}~d\xi dxdt\\
=&\int_{0}^{T}\int_{\mathbb{T}^{3}\times\mathbb{R}^{3}} |\xi-u^{\epsilon}_{f}+u^{\epsilon}_{f}-u_{f}|^{2}f^{\epsilon}~d\xi dxdt\\
\leq&2\int_{0}^{T}\int_{\mathbb{T}^{3}\times\mathbb{R}^{3}}(|\xi-u^{\epsilon}_{f}|^{2}+|u^{\epsilon}_{f}-u_{f}|^{2})f^{\epsilon}~d\xi dxdt\\
\leq&C\epsilon+\int_{0}^{T}\int_{\mathbb{T}^{3}}\mathcal{H}(U^{\epsilon}|U)~dxdt\\
\leq&C\sqrt{\epsilon}.
\end{align*}
Since $\rho_{f}^{\epsilon}\longrightarrow \rho_{f} \  \mathrm{in}  \   L^{2}(0,T; H^{-1}(\mathbb{T}^{3}))$, then  $A_{2}\rightarrow 0$ as $\epsilon\rightarrow 0$. Thus, we have
\begin{equation}
{f}^{\epsilon}\longrightarrow \rho_{f}\otimes\delta_{\xi=u_{f}} \ \ \ \mathrm{in} \ \ \  \mathcal{D}'((0,T)\times\mathbb{T}^{3}\times\mathbb{R}^{3}).
\end{equation}


\section*{Declarations}
\noindent {\textbf{A conflict of interest statement} }  On behalf of all authors, the corresponding author states that there is no conflict of interest.

\noindent {\textbf{Data availability statement} } Data sharing is not applicable to this article as no datasets were generated or analyzed during the current study.

\end{document}